\documentclass[11pt]{amsart}
\usepackage{amsmath,fullpage}
\usepackage[spacing, kerning]{microtype}
\usepackage{amsfonts,mathrsfs, soul}
\usepackage{amsthm}
\usepackage{amssymb, mathtools}
\usepackage{enumerate}
\usepackage{tikz-cd}

\usepackage[bookmarksnumbered,linktocpage,hypertexnames=false,colorlinks=true,linkcolor=blue,urlcolor=blue,citecolor=purple,anchorcolor=green,breaklinks=true]{hyperref}
\usepackage{cleveref}
\usepackage[abbrev, alphabetic, nobysame]{amsrefs}
\usepackage{comment}
\usepackage{enumitem}
\theoremstyle{definition}
\newtheorem{Def}{Definition}[section]
\newtheorem{ex}[Def]{Example}
\newtheorem{rem}[Def]{Remark}

\theoremstyle{plain}
\newtheorem{prop}[Def]{Proposition}

\newtheorem{theorem}[Def]{Theorem}
\newtheorem{Thm}[Def]{Theorem}

\newtheorem{Conj}[Def]{Conjecture}
\newtheorem{question}[Def]{Question}
\newtheorem*{thm*}{Theorem}
\newtheorem{lem}[Def]{Lemma}
\newtheorem{cor}[Def]{Corollary}
\newtheorem*{cor*}{Corollary}

\newtheorem*{con*}{Conjecture}
\newtheorem*{frag*}{Question}
\newtheorem*{verm*}{Vermutung}

\usepackage{mathrsfs}

\newcommand\wt{\widetilde}

\newcommand\sO{\mathscr O}
\newcommand\sI{\mathscr I}

\renewcommand\P{\mathbb{P}}

\newcommand{\C}{{\mathbb C}}
\newcommand{\R}{{\mathbb R}}

\newcommand{\N}{{\mathbb N}}
\newcommand{\Z}{{\mathbb Z}}

\title[Short Title]{Stubborn Polynomials}
\author{Lorenzo Baldi}
\address{Universit\"at Leipzig, Leipzig, Germany}
\email{lorenzo.baldi@uni-leipzig.de}
\author{Grigoriy Blekherman}
\address{Georgia Institute of Technology, Atlanta, Georgia, USA}
\email{greg@math.gatech.edu}
\author{Khazhgali Kozhasov}
\address{Universit\'e C\^ote d'Azur, Nice, France}
\email{khazhgali.kozhasov@univ-cotedazur.fr}
\author{ \\ Daniel Plaumann}
\address{Technische Universität Dortmund, Dortmund, Germany}
\email{daniel.plaumann@math.tu-dortmund.de}
\author{Bruce Reznick}
\address{University of Illinois at Urbana-Champaign, Urbana, Illinois, USA}
\email{reznick@illinois.edu}
\author{Rainer Sinn}
\address{Universit\"at Leipzig, Leipzig, Germany}
\email{rainer.sinn@uni-leipzig.de}
\date{\today}
\thanks{L. Baldi was funded by the Humboldt Research Fellowship for postdoctoral researchers.}
\newcommand\sV{\mathcal{V}}
\newcommand\wh{\widehat}

\newcommand\ol{\overline}

\renewcommand{\div}{\mathrm{div}}
\DeclareMathOperator{\Pic}{Pic}

\begin{document}


\begin{abstract}
The relationship between nonnegative polynomials and sums of squares is a classical topic in real algebraic geometry. We study \emph{stubborn polynomials} $f$ on a real variety $X$, which are polynomials nonnegative on $X$, such that no odd power of $f$ is a sum of squares. Previously, stubborn polynomials were studied only in the globally nonnegative case, with results restricted to polynomials nonnegative on $\mathbb{P}^2$. We fully characterize stubborn polynomials on smooth curves, showing that a polynomial on a smooth totally real curve is stubborn if and only if all of its zeros are real. This implies that there exist smooth curves with no stubborn polynomials in low degree, while stubborn polynomials must exist in sufficiently high degrees on curves with positive genus.
We explore the much more delicate situation with singular and reducible curves. While being real-rooted always implies being stubborn, there also exist singular curves with no stubborn polynomials at all. 

We then analyze the case of ternary sextics, i.e.,~polynomials of degree $6$ on $\mathbb{P}^2$. We prove the Conjecture of Blekherman, Kozhasov, and Reznick that a nonnegative ternary sextic is stubborn if and only if its real delta-invariant is at least 9. To analyze this case, we develop results for lifting stubborn polynomials from curves to higher dimensional varieties and use the theory of weak Del Pezzo surfaces. We complement these results with structural properties of stubborn polynomials and present many explicit examples.
\end{abstract}

\maketitle

\section{Introduction}

The study of nonnegative polynomials and their relationship with sums of squares is a fundamental topic in real algebraic geometry. Hilbert showed \cite{MR1510517} that there exist nonnegative polynomials that are not sums of squares of polynomials. Hilbert's 17th problem \cite{MR1557926}, answered in the affirmative by Artin \cite{MR3069468}, asked whether every nonnegative polynomial is a sum of squares of rational functions. We are interested in understanding a certificate of nonnegativity that lies between being a sum of squares of polynomials and a sum of squares of rational functions: what are nonnegative polynomials $f$ such that an odd power $f^{2k+1}$ is a sum of squares of polynomials? Polynomials $f$ such that no odd power $f^{2k+1}$ is a sum of squares are called \emph{stubborn}. 

The systematic study of stubborn polynomials was initiated in \cite{Stubborn2024} for globally nonnegative polynomials (that is, $f$ nonnegative on the ambient space). Prior to that work only isolated examples were known, see \cite{MR554130}. We investigate stubborn polynomials on \emph{real projective varieties}. As in the case of globally nonnegative polynomials, we find that stubbornness of nonnegative polynomials on varieties is closely related to the geometry of their real zeroes. In this work we significantly improve our understanding of this relation.

Our first result is about polynomials on irreducible projective curves. For  \emph{smooth projective curves}, we get a surprising characterization of stubborn polynomials in terms of their roots. We say that a form (that is, a homogeneous polynomial) $F$ on a real curve $C$ is \emph{real-rooted} if the hypersurface defined by the vanishing of $F$ intersects $C$ only in real points. In the sequel, a variety $X$ is said to be \emph{totally real} if $X(\R)$ is Zariski dense in $X(\C)$. 

\begin{Thm}[{see \Cref{thm:many_zeroes_implies_stubborn,thm:stubborn_implies_many_zeroes,thm:large degree}}]
    \label{thm:main-curves}
    Let $X \subset \P^n$ be an irreducible, totally real, projective $d$-normal curve. Let $F \in \R[X]_{2d}$ be nonnegative on $X(\R)$ but not a sum of squares. Consider the following conditions:
    \begin{enumerate}
        \item $F$ is real rooted on $X$ with zeroes at smooth points;
        \item $F$ is stubborn.
    \end{enumerate}
    Then $(1) \implies (2)$, and if $X$ is smooth, then $(2) \implies (1)$. Furthermore, if $X$ is smooth of genus $g\ge 1$ and $d$ is large enough, there exist stubborn forms in $\R[X]_{2d}$.
\end{Thm}

For the definition of $d$-normality see Section \ref{sec:def}. For projectively normal (and therefore smooth) curves, the above theorem implies that $F$ is stubborn on $X$ if and only if it is real-rooted.
The proof of the equivalence (1) $\iff$ (2) in this case relies on results from \cite{Scheiderer2011, baldiblekhermannsinn24}. It has unexpected consequences. 
First, only special \emph{extreme rays} of the cone $P_{X,2d}$ of nonnegative forms of degree $2d$ on $X$ (which are not squares) can be stubborn. This is different from the conjectural situation for $\mathbb{P}^2$ (\cite[Conj. $6$]{Stubborn2024}), where it is conjectured that every extreme ray of the cone of nonnegative polynomials (which is not a square) is stubborn.
Second, this equivalence combined with \cite[Cor.~3.4.10]{baldiblekhermannsinn24}, shows that for sufficently large $d$, the cone $P_{X,2d}$ contains stubborn forms for any smooth real curve $X$, see the second part of Theorem \ref{thm:main-curves}. However, there also exist smooth real curves on which there are no stubborn forms of small degrees $d$ at all! We give several constructions of such curves in Section \ref{sec:nostub}.
Based on this observation, for a given totally real projective variety $X$, we call an even degree $2d$ \emph{stubborn} if there exist stubborn polynomials of $X$ of degree $2d$. While real projective varieties on which all nonnegative forms in a fixed degree are sums of squares have been classified \cite{MR3486176}, our result opens a door for an intriguing research direction: understand real varieties $X$ with non-stubborn degrees.

We also investigate the existence of stubborn polynomials in the much more subtle case of non-smooth, and potentially reducible, real curves in Section~\ref{sec:cubics}. To illustrate the issues that can arise, we give the full ``zoo" of plane cubics, and explain in which cases there exist stubborn quadrics. In the general case, stubborn forms are connected to $2$-torsion points in the generalized Jacobian of the curve, see Section~\ref{sec:2torsion}.

Next, we look at stubborn polynomials on \emph{surfaces}. The paper \cite{Stubborn2024} introduced the invariant $\delta^{\R}$ of a real singularity of a plane curve. In \cite[Cor. 27]{Stubborn2024} it was shown that a nonnegative form on $\P^2$ with ``many" real singularities is necessarily stubborn. In particular, any nonnegative ternary sextic (form of degree $6$ on $\P^2$) with 10 real zeroes (counted by adding the local $\delta^\R$-invariants) is stubborn. We significantly generalize this result by completely classifying stubbornness of ternary sextics, solving a conjecture from \cite{Stubborn2024} {in the process}.

\begin{Thm}[{see Theorems \ref{thm:stubbornSextics} and \ref{thm:delta_geq_9}}]
\label{thm:sextics}
Let $F \in \R[x,y,z]_6$ be a nonnegative ternary sextic which is not a sum of squares. Then $F$ is stubborn if and only if $\delta^{\R}(F)\geq 9$.
\end{Thm}

The proof that the inequality $\delta^{\R}(F)\geq 9$ implies that $f$ is stubborn makes use of our curve results. Such a form $F$ \emph{inherits} stubbornness globally on $\P^2$ from stubbornness on a cubic curve! The proof that $\delta^{\R}(F)\leq 8$ implies that $F$ is not stubborn uses the classical theory of \emph{weak Del Pezzo} surfaces and a projective Positivstellensatz by Scheiderer \cite{Scheiderer2011}.

Finally, inspired by the situation with ternary sextics, we investigate when nonnegative polynomials on a curve can be ``lifted" to be nonnegative on a larger variety. If we can lift a nonnegative real-rooted form on a curve, this gives us a way of constructing stubborn forms on higher dimensional varieties, as the lifted form inherits stubbornness (see Proposition \ref{prop:inherit_stubbornness}).
Among the results in \Cref{sec:lifting}, we show, for instance, that lifting is often possible from elliptic normal curves. 

\begin{theorem}[see \Cref{thm:general_lifting} and \Cref{cor:ext}]\label{cor:lift}
Let $X\subseteq \P^n$ be a totally real variety whose vanishing ideal $I(X)$ is generated by quadrics. Let $L\subset \P^n$ be a linear space, such that $C=L\cap X$ is an elliptic normal curve in $L$, and $T_p C=T_p X\cap L$ for all $p\in C$.
Then for any quadric $F\in \R[C]_2$ nonnegative on $C$, there is a quadratic form $\hat{F}\in \R[X]_2$ nonnegative on $X$ such that $\hat{F} \mid_C=F$, i.e.,~$F$ can be extended to a quadratic form nonnegative on $X$.
\end{theorem}

It was shown in \cite{MR3486176} that if $X$ is a totally real variety, then every nonnegative quadratic form in $\R[X]_2$ is a sum of squares of linear forms if and only if $X$ is a variety of minimal degree. We can equivalently state this result as follows: if $X$ is a totally real variety, then every nonnegative quadratic form in $\R[X]_2$ can be extended to a globally nonnegative quadratic form on $\mathbb{P}^n$ if and only if $X$ is a variety of minimal degree. This interpretation naturally leads to the following question:

\begin{question}
Describe pairs of totally real varieties $X, Y$ with $X\subset Y$ such that any nonnegative quadratic form on $X$ can be extended to a nonnegative quadric on $Y$.
\end{question}

It follows from the result of \cite{MR3486176} that for a totally real variety of minimal degree $X$ any nonnegative quadratic form on $X$ can be extended to be nonnegative on $Y$ for any $Y$ containing $X$. However, we are not aware of any other results on pairs of totally real varieties where such an extension is possible. \Cref{cor:lift} gives us the first instance of lifting nonnegativity beyond varieties of minimal degree, as it allows us to lift nonnegative quadrics from elliptic normal curves.

Our next result says that on an elliptic normal curve, nonnegative forms of higher degree can be always extended to globally nonnegative forms, which together with Theorem \ref{thm:main-curves} gives rise to stubborn forms with ``few" real zeroes in $\P^n$.

\begin{theorem}[see \Cref{cor:ext_elliptic}]
    \label{cor:elliptic_normal_Pn}
Let $C$ be an elliptic normal curve of degree $n$ in $\P^n$. Then any form of degree $2d$  ($d\geq 3$ for $n=2$, and $d\geq 2$ for $n\geq 3$) nonnegative on $C$, can be extended to a globally nonnegative form on $\P^n$. In particular, there exist stubborn forms of degree $2d$ on $\P^n$ with at most $dn$ real zeroes, where each real zero is a simple node.
\end{theorem}

We complement these results with structural properties of stubborn polynomials. For instance, in Section \ref{sec:def} we show that non-stubborn forms of degree $2d$ on $X$ form a convex subcone of the cone $P_{X,2d}$ of nonnegative forms of degree $2d$ on $X$, generalizing a result of \cite{Stubborn2024}.
\medskip

We discuss our results in the context of the first example of a stubborn form given by Stengle. 
\begin{ex}
Stengle showed in \cite{MR554130} that the ternary sextic
$$S(x,y,z)=x^3z^3+(y^2z-x^3-z^2x)^2$$
is stubborn. The form $H = y^2z-x^3-z^2x=0$ defines a smooth cubic curve $C$ in $\mathbb{P}^2$. Stengle's argument that $S$ is stubborn is somewhat ad-hoc. Using our results, we can show that $S$ is stubborn using the curve $C$: The sextic $S$ intersects $C$ in $[0:0:1]$ with multiplicity 6 and $[0:1:0]$ with multiplicity 12, because $S$ reduces modulo the ideal $I(H)$ of $C$ to $x^3z^3$. Since $S(x,y,z)$ is globally nonnegative, as shown by Stengle, it follows that $S(x,y,z)$ is stubborn on $C$ by Theorem~\ref{thm:main-curves}. So $S$ inherits its stubbornness on $\P^2$ from being stubborn on $C$ (this was also observed by Stengle).
Moreover, this inheritance is not an artifact of Stengle's construction: as we show in the proof of \Cref{thm:sextics}, stubbornness of any ternary sextic is inherited from its stubbornness on some cubic curve.
Finally, as we show in \Cref{cor:lift}, we can extend any form nonnegative on $C$ to a globally nonnegative form, so Stengle's construction is not accidental.
\end{ex}

\section{Stubborn polynomials on varieties}\label{sec:def}

A \emph{real variety} is a reduced separated scheme of finite type over $\R$ (note that we do not assume irreducibility or smoothness). A main source of projective real varieties are the zero sets in complex projective space of real homogeneous polynomials.  
A \emph{totally real variety} is a real variety whose set of real points $X(\R)$ is Zariski dense in $X(\C)$. For irreducible varieties this is equivalent to the existence of a smooth real point in $X$ by the Artin-Lang Theorem, see \cite[Theorem~1.7.8]{Scheiderer2024}.

If $X\subset \P^n$ is an embedded real variety, we denote by $\R[X]$ its homogeneous coordinate ring. An even-degree form $F \in \R[X]_{2d}$ has a well-defined sign at each $p\in X(\R)$. We say that $F$ is nonnegative on $X(\R)$ if $F(p) \ge 0$ for all $p \in X(\R)$; in this case, we simply write $F \ge 0$ on $X(\R)$. A particularly simple class of nonnegative forms inside $\R[X]_{2d}$ are sums of squares, i.e., forms $F$ which can be written as $F=G_1^2 + \dots + G_m^2$ for some $G_i\in \R[X]_d$.

We study nonnegative forms that are not sums of squares, and moreover that do not become sums of squares under taking arbitrary odd powers.
\begin{Def}
    \label{def:stubborn}
  Let $X\subset \P^n$ be a totally real variety with homogeneous coordinate ring $\R[X]$. An even-degree form $F \in \R[X]$ is \emph{stubborn} if it is nonnegative on $X$ and $F^k$ is not a sum of squares for any odd $k\in \Z_{>0}$. An even degree $2d$ is called \emph{stubborn} for $X$ if there exist stubborn polynomials of degree $2d$ on $X$.
\end{Def}

Nonnegative forms of degree $2d$ on a totally real variety $X$ form a closed, full-dimensional, pointed cone in the real vector space $\R[X]_{2d}$ (see e.g.~\cite{MR3486176}), which we denote by $P_{X,2d}$. It is often convenient to consider only quadratic forms $\R[X]_2$, which can be achieved by the $d$-th Veronese embedding of $X$: set $Y = \nu_d(X)$ so that $P_{Y,2}=P_{X,2d}$. In the case $2d=2$ we drop the subscript and denote the cone of nonnegative quadratic forms on $X$ by $P_X$.

We now show that for a totally real variety $X$ the set of \emph{non-stubborn} polynomials of degree $2$ on $X$ is also a \emph{convex cone}, which is a subcone of $P_X$. We adapt the proof of the analogous result \cite[Theorem~44]{Stubborn2024} for the Veronese variety $X=\nu_d(\P^n)$.

\begin{theorem}\label{thm:convexity}
Let $X\subset \P^n$ be a totally real projective variety. If $F_1, F_2 \in P_X$ and $F_1^{k_1}$, $F_2^{k_2}$ are sums of squares for some odd $k_1, k_2\in \mathbb{N}$, then $(F_1+F_2)^{k_1+k_2-1}$ is a sum of squares.
  In particular, the set of non-stubborn quadrics is a convex cone.
\end{theorem}

\begin{proof}
For integers $\ell\geq 2r\geq 0$ let
\begin{align*}
  b_{\ell,2r}(t)\ =\ \sum_{i=0}^{2r} \binom{\ell}i\, t^i
\end{align*}
be the truncated binomial polynomial. By \cite[Thm. 43]{Stubborn2024}, if $\ell>2r$ then $b_{\ell,2r}(t)>0$ for all $t\in \R$.
Therefore, there are polynomials $g_{\ell,2r}(t)$ and $h_{\ell,2r}(t)$ 
with $b_{\ell,2r}(t) = (g_{\ell,2r}(t))^2 +  (h_{\ell,2r}(t))^2$, and (by homogenization) there exist binary forms $G_{\ell, 2r}, H_{\ell, 2r}$
of degree $r$ such that
\[
B_{\ell,2r}(t_1,t_2)\ =\ \sum_{i=0}^{2r} \binom {\ell}i\, t_1^it_2^{2r-i}\ =\  (G_{\ell,2r}(t_1,t_2))^2 +  (H_{\ell,2r}(t_1,t_2))^2,
\]
where $B_{\ell, 2r}$ is the homogenization of $b_{\ell, 2r}$.
We now expand $(F_1 + F_2)^{k_1+k_2-1}$, and by letting $i' = i - k_1$ in the following, we obtain its desired representation as a sum of squares,

\begin{align*}
(F_1+F_2)^{k_1+k_2-1} &= \sum_{i = 0}^{k_1+k_2-1} \binom{k_1+k_2 -1}i F_1^{i}F_2^{\,k_1+k_2-1-i}\\ 
&=\ \sum_{i = 0}^{k_1-1} \binom{k_1+k_2-1}i F_1^{i}F_2^{k_1+k_2-1-i}  +  \sum_{i = k_1}^{k_1+k_2-1} \binom{k_1+k_2-1}i F_1^{i}F_2^{k_1+k_2-1-i}\\
&=\ F_2^{k_2} \sum_{i = 0}^{k_1-1} \binom{k_1+k_2-1}i F_1^{i}F_2^{k_1-1-i}  +
 F_1^{k_1}  \sum_{i' = 0}^{k_2-1} \binom{k_1+k_2-1}{i '+k_1}F_1^{i'}F_2^{k_2-1-i'}\\
\end{align*}
\begin{align*}
&=\ F_2^{k_2} \sum_{i = 0}^{k_1-1} \binom{k_1+k_2-1}i F_1^{i}F_2^{k_1-1-i}  +
 F_1^{k_1}  \sum_{i' = 0}^{k_2-1} \binom{k_1+k_2-1}{k_2-1-i'}F_1^{i'}F_2^{k_2-1-i'}\\
     &=\ F_2^{k_2}B_{k_1+k_2-1,k_1-1}(F_1,F_2) + F_1^{k_1}B_{k_1+k_2-1,k_2-1}(F_2,F_1) \\
 &=\  F_2^{k_2}(G_{k_1+k_2-1,k_1-1}^2(F_1,F_2) +  H_{k_1+k_2-1,k_1-1}^2( F_1,  F_2)) \\
 &\ \ \ \,+ F_1^{k_1}(G^2_{k_1+k_2-1,k_2-1}(F_2,F_1) +  H_{k_1+k_2-1,k_2-1}^2(F_2,F_1)).
\end{align*}
The last expression is a sum of squares, since the cone of sums of squares is closed under products.
It follows in particular that the set of non-stubborn polynomials is a convex cone.
\end{proof}

A result of Scheiderer \cite{Scheiderer2011} shows that if $X$ is a smooth totally real curve or a totally real variety $X\subset\mathbb{P}^n$ of dimension at least 2, then all strictly positive forms on $X$ are non-stubborn. Therefore, in these cases, non-stubborn polynomials include the entire interior of $P_X$. 
\begin{cor}
    The set of non-stubborn quadrics on a smooth totally real variety $X \subset \P^n$ is a convex cone, containing the strictly positive quadrics on $X(\R)$. \qed
\end{cor}

We conclude this section with a discussion of what existence of stubborn polynomials in degree $2d$ implies for existence of stubborn polynomials in higher degree. We first discuss the notion of \emph{$d$-normality} and \emph{projective normality} which we need in the sequel.

\begin{Def}
    \label{def:dnormal}
    Let $d > 0$ be a positive integer. A projective variety (not necessarily smooth or irreducible) $X \subset \P^n$ is called $d$-\emph{normal}  if the natural restriction morphism $H^0(\P^n, \sO_{\P^n}(d)) \to H^0(X, \sO_{X}(d))$ is surjective. A variety $X \subset \P^n$ is called \emph{projectively normal} if it is normal and $d$-normal for all $d$.
\end{Def}

Consider the short exact sequence of sheaves
\[
    0 \longrightarrow \sI_X \longrightarrow \sO_{\P^n} \longrightarrow \sO_X \longrightarrow 0
\]
Twisting by $d$ and taking cohomology, we see that $X$ is $d$-normal if and only if $H^1(\P^n, \sI_X(d)) = 0$. It follows from Serre's Vanishing theorem \cite[Ch. III, Thm.~5.2]{zbMATH03572315} that $X \subset \P^n$ is $d$-normal for all sufficiently large $d$. Moreover, since $\sI_X(d) \cong \sI_{\nu_d(X)}(1)$, it also follows that, if we take a sufficiently high Veronese reembedding, we obtain a $1$-normal variety.

Suppose that $X$ is projectively normal. Let $X'$ be a general hyperplane section of $X\subset \P^n$. We have exact sequence of sheaves (where $\sI_{X'}$ is a sheaf on $\P^{n-1}$):
$$0 \rightarrow \sI_X (k-1) \rightarrow \sI_X (k) \rightarrow \sI_{X'}(k) \rightarrow 0.$$
This leads to a long exact cohomology sequence
$$0 \rightarrow H^0(\P^n, \sI_X (k-1)) \rightarrow H^0(\P^n, \sI_X (k)) \rightarrow H^0(\P^{n-1}, \sI_{X'}(k))
\rightarrow H^1(\P^n, \sI_X (k-1))\rightarrow\ldots .$$
Since $H^1(\P^n, \sI_X (k-1))=0$, we get that any form of degree $k$ vanishing on $X'$ can be lifted to a form of degree $k$ on $\P^n$ containing $X$.

We now use the above discussion to show an implication for existence of stubborn polynomials.

\begin{prop}
Let $X$ be a totally-real, irreducible, projectively normal variety of dimension at least $2$ and suppose that $F\in \R[X]_{2d}$ is stubborn. Then all degrees greater than or equal to $2d$ are stubborn for $X$.
\end{prop}
\begin{proof}
Let $H$ be a real hyperplane defined by the vanishing of a linear form $L$, which intersects $X$ transversely with a real intersection point such that $Y = H\cap X$ is irreducible. 
Then $Y$ is an irreducible totally real variety and $I(Y)=L+I(X)$ (see the paragraph before the proposition). We claim that $L^{2m}F$ is stubborn for all $m\geq 0$, which we prove by induction on $m$. The base case $m=0$ is given. Suppose that $(L^{2m+2}F)^k=\sum G_i^2$ for some odd $k$. Then, each $G_i$ vanishes on $L$, and therefore $L$ divides each $G_i$. We apply this observation repeatedly to see that $L^k$ divides each $G_i$. Thus we see that $(L^{2m}F)^k$ is a sum of squares, which is a contradiction.
\end{proof}
 We note that the above argument fails for curves, and we do not have a general statement for this case. However, as we will see below, stubborn polynomials exist in all degrees on elliptic normal curves, and for all polynomials of sufficiently large degree on smooth totally real curves.

\section{Existence of stubborn polynomials on curves}\label{sec:curves}

In this section we investigate stubborn polynomials on projective curves.  
We show below that these are closely linked to nonnegative forms with only real zeroes.

\begin{Def} Let $X\subset \P^n$ be an irreducible totally real projective curve. We say that a form $F\neq 0$ in $\R[X]$ is \emph{real-rooted} if all of the finitely many zeroes of $F$ on $X$ are real.
\end{Def}

In this section a curve $X\subset \P^n$ is always assumed to be totally real.
For $F \in \R[X]_{2d}$ we write $\div(F)$ for its divisor of zeros. For any divisor $D = \sum_i n_i P_i$ on $X$, its real part is denoted by $D_{\R} = \sum_{i\, :\,P_i \in X(\R)} n_i P_i$. Then $F$ is real-rooted if and only if $\div(F) = \div(F)_\R$.

Recall now the definition of $d$-normal variety from \Cref{def:dnormal}. We prove a necessary condition for stubbornness on curves.

\begin{theorem}
\label{thm:many_zeroes_implies_stubborn}
    Let $X\subset \P^n$ be a totally real projective curve. Let $F \in \R[X]_{2d}$ be nonnegative on $X(\R)$, but not a sum of squares. If $X$ is $d$-normal and $F$ real-rooted with all zeros at smooth points of $X$, then $F$ is stubborn.
\end{theorem}
\begin{proof}
    We denote by $J=\mathrm{Pic}^0(X)$ the generalized Jacobian of $X$, see \cite[p.~249]{harrisModuliCurves1998}. Its points correspond to linear equivalence classes of divisors on $X$ supported on smooth points of $X$ \cite[p.~251]{harrisModuliCurves1998}. By Bertini's theorem (see e.g. \cite[Thm.~II.8.18]{zbMATH03572315}), we can choose a real hyperplane that intersects $X$ transversely in smooth (not necessarily real) points, and let $H$ be the intersection divisor.
    
    We have $[2dH - \div(F)] = 0$ in $J$, and since $F$ is nonnegative on $X$ and real-rooted we see that $\eta=d[H] - \frac{1}{2}[\div(F)] \in J(\R)_2$ is a real $2$-torsion point. Since $F$ is not a square and $X$ is $d$-normal, it follows that $\eta \neq 0$, and thus 
    \[ (2k+1)d[H] - \frac{1}{2}[\div(F^{2k+1})] = (2k+1)\eta = \eta\] 
    is also a non-trivial $2$-torsion point of the generalized Jacobian. Therefore, $F^{2k+1}$ is not a square. 
    
   If we could write $F^{2k+1} = \sum_{i=1}^m G_i^2$, then since $F$ is nonnegative and real-rooted on $X$, we would have that for all $i$, $G_i^2$ is equal (up to a positive constant) to $F^{2k+1}$. This would imply that $F^{2k+1}$ is a square, which is a contradiction. Hence $F^{2k+1}$ is not a sum of squares for all $k$, i.e., it is stubborn. 
\end{proof}

It follows from Max Noether's AF+BG theorem that plane curves are $d$-normal for all $d$ \cite[Sec.~5.5]{fultonAlgebraicCurves1989}. We therefore get the following result from \Cref{thm:many_zeroes_implies_stubborn}.

\begin{cor}\label{cor:plane_cubics}
    Let $X \subset \P^2$ be a totally real projective plane curve. Let $F \in \R[X]_{2d}$ be nonnegative on $X(\R)$, but not a sum of squares. If $F$ is real-rooted with all zeros at smooth points of $X$, then $F$ is stubborn. \qed
\end{cor}

For smooth curves, the reverse implication also holds, thus proving that stubbornness and real-rootedness are equivalent for nonnegative forms that are not sums of squares.

\begin{theorem}
    \label{thm:stubborn_implies_many_zeroes}
    Let $X\subset \P^n$ be an irreducible smooth real projective curve. If $F\in \R[X]_{2d}$ is stubborn, then all zeroes of $F$ on $X$ are real.
\end{theorem}

\begin{proof}
    Suppose that $F$ is not real rooted.
    Let $H$ be a hyperplane divisor, and define the divisor $D = dH - \frac{1}{2} \div(F)_\R$. 
    Since $F$ is not real rooted, $\deg(D) > 0$ and so $D$ is ample. Pick an odd $k\in \N$ such that $kD$ is very ample and let $L$ be a section of $\sO_X(k \frac12 \cdot \div(F)_\R)$.
    Then $F^{k}/L^{2} \in H^0(X, \sO_X(k D)^{\otimes 2}) = H^0(X, \sO_X(2k D))$ because the divisor of zeroes of the section $F^{k}/L^{2}$ is linearly equivalent to $k \cdot \div(F) - 2 \cdot \frac12 \cdot k \cdot \div(F)_\R = k\cdot (\div(F) - \div(F)_\R) = k \cdot 2 \cdot D$ (see e.g. \cite[II~Prop.~7.7]{zbMATH03572315}). Since $\div(F) - \div(F)_\R$ has no real point, the section $F^{k} /L^{2}$ is strictly positive on $X(\R)$. It follows from \cite[Theorem.~4.11]{Scheiderer2011} that $(F^{k}/L^{2})^\ell \in H^0(X, \sO_X(2k\ell D))$ is a sum of squares  $\sum_i G_i^2$ for all sufficiently large $\ell$. 
    
    We now want to conclude that some odd power of $F$ is a sum of squares. Clearing denominators above, we get
    \[ F^{k\ell} = \sum_i G_i^2 L^{2\ell} \in H^0(X,\sO_X(2k\ell d H)),  \]
    and the divisor of each summand is $k\ell d H$. Choosing $\ell$ odd and sufficiently large such that $X$ is $k \ell d$-normal, the sections $G_iL$ correspond to elements $\widetilde{G_i}\in \R[X]_{k\ell d}$ so that $F = \sum_i \widetilde{G_i}^2$.
\end{proof}

We now show that stubborn polynomials always exist in sufficiently large degree for smooth curves of positive genus. Smooth rational curves have no stubborn polynomials of any degree, essentially because every nonnegative binary form is a sum of squares, see the discussion in Section \ref{sec:smoothnostub} for details.
\begin{theorem}\label{thm:large degree}
    Let $X\subset \P^n$ be an irreducible smooth totally real projective curve of genus $g \ge 1$. For all sufficiently large $d$ there exists $F \in \R[X]_{2d}$ which is stubborn.
\end{theorem}
\begin{proof}
    Since $g \ge 1$, it follows from \cite[Cor. 3.4.10]{baldiblekhermannsinn24} that for all sufficiently large $d$, there exists a real-rooted nonnegative form $F_d \in \R[X]_{2d}$ of degree $2d$ which is not a sum of squares. Moreover, for all sufficiently large $d$ we have $d$-normality, and therefore $F_d$ is stubborn by \Cref{thm:many_zeroes_implies_stubborn}. 
\end{proof}

We remark that in some of the previous theorems, we have assumed the curves to be irreducible. The added technical difficulty for reducible curves is that one can have stubborn polynomials vanishing on an irreducible component.
We study examples of this in the next subsection. 

\subsection{Stubborn polynomials on plane cubics}\label{sec:cubics}

The main point we would like to illustrate is that existence of stubborn polynomials on singular curves is a quite delicate question. To this end, we fully characterize existence of low degree stubborn forms on singular and possibly reducible plane cubics. 

\begin{ex}
    If $X$ is smooth, it is not hard to show directly that there exists an irreducible indefinite quadratic form $F$ that is nonnegative on $X$ with three real zeros of multiplicity $2$ (see also \cite[Prop.~4.1.1]{baldiblekhermannsinn24}), as shown in Fig.~\ref{fig:stubborn_on_irreducible_cubic}. By Theorem~\ref{thm:many_zeroes_implies_stubborn}, any such $F$ is stubborn.
\end{ex}

    \begin{figure}\label{pic:cub0}
        \includegraphics[width=8cm]{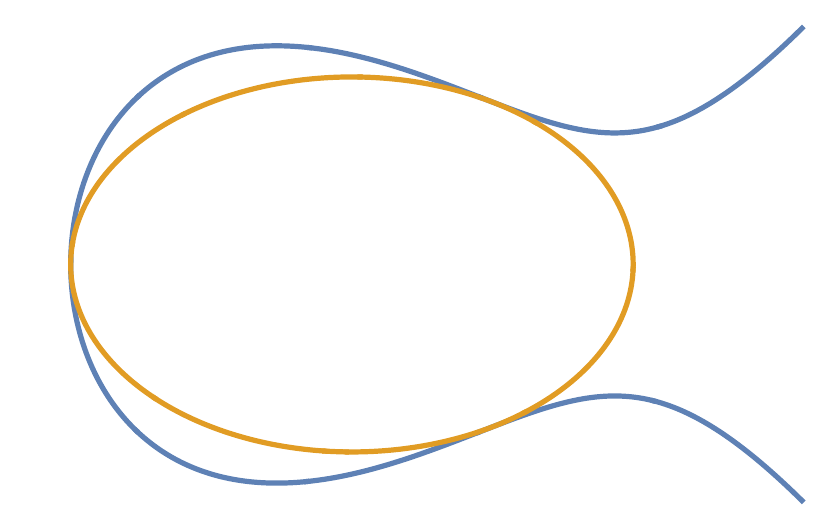}
    \caption{A stubborn quadratic on an irreducible cubic.}\label{fig:stubborn_on_irreducible_cubic}
    \end{figure}

\begin{ex}
    Next, we look at the irreducible singular cubics. There are three cases.

    (a) If $X$ is a nodal cubic with connected real locus, we may assume that $X$ is defined by $y^2z-x^2(x+z)=0$ after a projective change of coordinates. In this case, there are stubborn quadratic forms with three real roots in non-singular points, as in the smooth case. 
    For example, $F=y^2-6x^2+3xz+9z^2$ is a quadratic form of this kind. 
    All such forms are stubborn by \Cref{cor:plane_cubics}.
    
    (b) A cuspidal cubic $X$, defined without loss of generality by $y^2z-x^3=0$, has no stubborn forms in any degree. To see this, let $X_0$ be the affine chart $z=1$ on $X$, which is parametrized by $t\mapsto (t^2,t^3)$. This identifies the coordinate ring $\R[X_0]=\R[x,y]/(y^2-x^3)$ with the subring \[\R[t^2,t^3]=\{f\in \R[t]\,\colon f'(0)=0\}\] of the univariate polynomial ring $\R[t]$. 
    So let $f\in\R[t]$ be a nonnegative polynomial with $f'(0)=0$. If also $f(0)=0$, then write $f^3=g_1^2+g_2^2$ as a sum of two squares in $\R[t]$. Since $f$ is divisible by $t^2$, $g_1$ and $g_2$ are also divisble by $t^2$, say $g_j=t^2\wt g_j$ for $j=1,2$, so that $f^3=(t^2\wt g_1)^2+(t^2\wt g_2)^2$ is a sum of two squares in $\R[X_0]$.

    Now assume that $f(0)\neq 0$, which means that the corresponding form in $\P^2$ does not pass through the singularity of $X$. Since the set of non-stubborn forms is convex by Theorem ~\ref{thm:convexity}, it suffices to show that no extremal nonnegative form is stubborn. Just as in the smooth case, it can be shown that such forms intersect $X$ in real points only, by \cite[Thm.~1.1]{KummerZalar2025}. Hence $f$ is real-rooted, so that $f=g^2$ for some $g\in\R[t]$. If $g'(0)=0$, then $f$ is a square in $\R[X_0]$. If $g'(0)\neq 0$, then, since $f'(0)=(g^2)'(0)=0$, we must have $g(0)=0$, so that $g$ is of the form $g=a_1t+a_2t^2+\cdots$. It follows that $g^k$ is in $\R[X_0]$ for all $k\ge 2$, so that $f^3=(g^3)^2$ is a square and $f$ is not stubborn. 

    (c) Let $X$ be a nodal cubic with a solitary node defined by $y^2z-x^2(x-z)=0$. The chart $X_0$ with $z=1$ is parametrized by $t\mapsto (t^2+1,t(t^2+1))$. The node corresponds to $t=\pm i$, so that the coordinate ring $\R[X_0]$ consists of those $f\in\R[t]$ that satisfy $f(i)=f(-i)=\ol{f(i)}$, 
    \[
        \R[X_0] = \{f\in\R[t]\,\colon f(i)\in\R\}.
    \]
    Again, we can apply \cite[Thm.~1.1]{KummerZalar2025} and need only consider real-rooted polynomials. 
    If $f\in\R[t]$ has only real roots and satisfies $f(i)\notin\R$ but $f^2(i)\in\R$ (i.e.,~if $f^2 \in \R[X_0]$ but $f \notin \R[X_0]$), then $f^2(i)<0$, so that $f^2$ is not nonnegative on $X_0$. We also need to look at polynomials of the form $(t^2+1)^r\cdot f^2$ with $f$ real-rooted. But all such polynomials are sums of squares as soon as $r\ge 2$, because for even $r\ge 2$ they are squares in $\R[X_0]$, and for odd $r\ge 3$, we have 
    \[(t^2+1)^r\cdot f^2=\bigl((t^2+1)^{\frac{r-1}{2}}\cdot f\bigr)^2\cdot (t^2+1)=\bigl((t^2+1)^{\frac{r-1}{2}}\cdot f\cdot t\bigr)^2+\bigl((t^2+1)^{\frac{r-1}{2}}\cdot f\bigr)^2\] which is a sum of two squares in $\R[X_0]$. (Note that for $r=1$, the summands may not lie in $\R[X_0]$; indeed, $t^2+1$ itself is not a sum of squares in $\R[X_0]$.) In particular, no such polynomial can be stubborn. 
\end{ex}

There are a number of cases to consider when $X$ is reducible.
    
\begin{ex}
    Suppose that $X=C\cup L$ is the union of a smooth conic $C$ defined by a non-degenerate quadratic form $Q$, and a line $L$ defined by a linear form $\ell$. Since $X$ is assumed to be totally real, $Q$ must be indefinite, so that $C$ is an ellipse. 

    (a) If $C$ and $L$ meet in two distinct real points, then $X$ admits a stubborn quadratic form. For example, the union of two distinct tangents to $C$ meeting $L$ in the same point defines such a quadratic form by \Cref{cor:plane_cubics}; see Fig.~\ref{fig:stubborn_on_circle_and_secant}. We can also verify this directly:
    let $p$ and $q$ denote the intersection points of $C$ and $L$, and let $\ell_1$ and $\ell_2$ be linear forms defining two tangents at points $a$ and $b$ on $C$ such that both $\ell_1$ and $\ell_2$ vanish on $L$ at the same point $c$. Then $\ell_1\ell_2$ is nonnegative on $X$. It is stubborn, because if we had $(\ell_1\ell_2)^k=\sum F_i^2$ in $\R[X]$ then each $F_i$, of degree $k$, would have to vanish at $a$, $b$ and $c$ to order $k$ and have no further zeros on $X$. After dehomogenizing with respect to a line not meeting $C(\R)$, as in the picture, the restriction of each $F_i$ to $C$ will have opposite signs in $p$ and $q$, while its restriction to $L$ will have the same sign at $p$ and $q$, a contradiction.
    \begin{figure}\label{pic:cub1}
        \includegraphics[width=8cm]{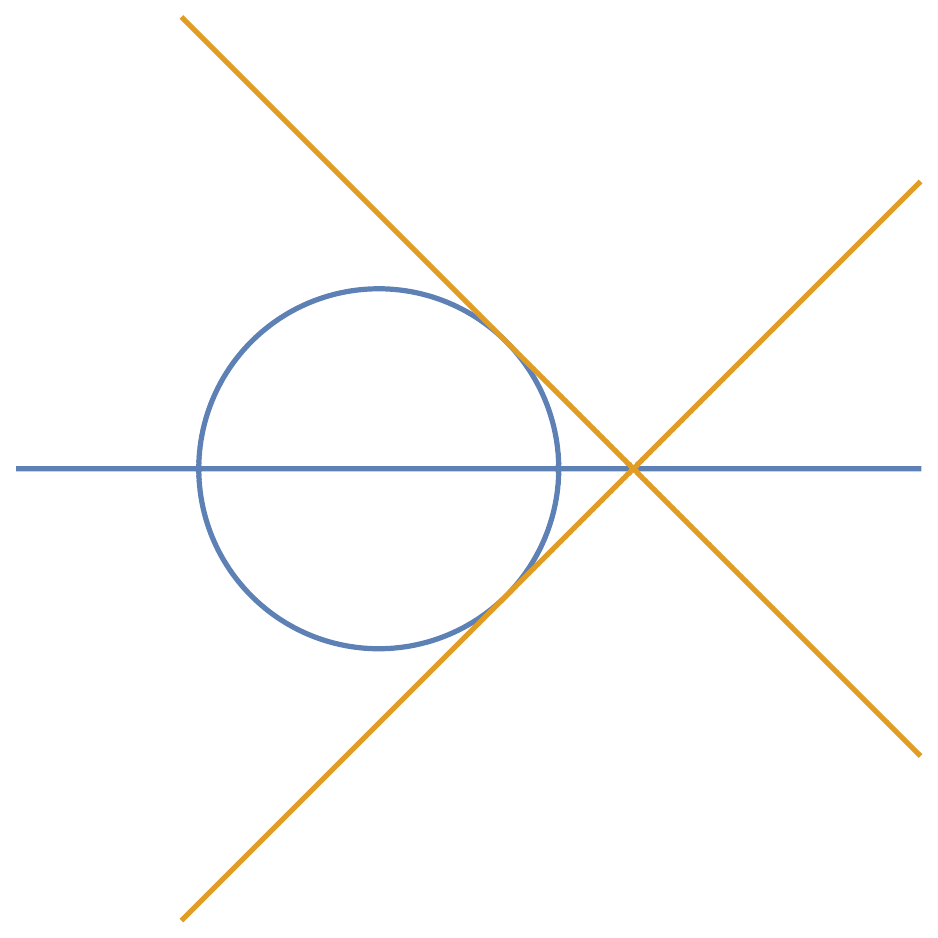}
    \caption{A stubborn quadric for a circle with a secant line.}\label{fig:stubborn_on_circle_and_secant}
    \end{figure}

    (b) If $L$ is tangent to $C$, then $X=C\cup L$ does not admit a stubborn quadratic form. We may assume after a change of coordinates that $Q=x^2+y^2-z^2$ and $\ell=z-x$. 
    A nonnegative quadratic form $F$ with only real zeros on $X$ is either a square of a linear form or it must be tangent to $L$ and $C$ at the intersection point $a=[1:0:1]$ and tangent to $C$ at a further real point $b$. (To see this, let $C'$ be the conic defined by $F$. Since $C'$ is tangent to $C$ at two points, so are the dual conics, hence $C$ and $C'$ share exactly two tangents, one of which must be $L$.)
    For any such $b$, there is a one-dimensional family of nonnegative quadratic forms, whose extreme points are $Q$ itself and the product of $\ell$ with the tangent line through $b$. We can verify this directly: After a change of coordinates, we may assume $b=[-1:0:1]$, then $F_0=Q$ and $F_1=(z-x)(z+x)$ span the family $\{F_\lambda=\lambda_0 F_0+\lambda_1 F_1\ |\ \lambda\in\P^1(\R)\}$ of all conics tangent to $C$ at $a$ and $b$. Here, $F_\lambda$ is nonnegative if and only $\lambda_0\lambda_1\ge 0$, so that $F_0$ and $F_1$ are the extreme points. (Note that $F_{[1,1]}=y^2$ is the line through $a$ and $b$ squared.) Neither $F_0$ nor $F_1$ is stubborn, so there are no stubborn quadratic forms by \Cref{thm:convexity}.
    
    However, $X$ does admit a stubborn form of degree $4$: there exists a quartic $F$ whose real locus consists of two ovals, each of which is tangent to $C$ at two points and tangent to $L$ at one point.  It can then be verified through symbolic computation that\\[-2em]

        {\footnotesize
    \begin{align*}
        F =& -\biggl(12 \sqrt{2}+5 \sqrt{936 \sqrt{2}+1337}+25\biggr) x^4-4 \biggl(15 \sqrt{2}+\sqrt{1872 \sqrt{2}+2674}+18\biggr) x^3 z\\ &+x^2 \biggl(\biggl(48 \sqrt{2}-2 \sqrt{936 \sqrt{2}+1337}+70\biggr) y^2-12 \biggl(2 \sqrt{2}+3\biggr) z^2\biggr)+8 \bigl(2 \sqrt{2}+3\bigr) x z \bigl(y^2-z^2\bigr)\\ &+\bigl(y^2-z^2\bigr) \biggl(\biggl(12 \sqrt{2}-\sqrt{936 \sqrt{2}+1337}+19\biggr) y^2+\biggl(36 \sqrt{2}+\sqrt{936 \sqrt{2}+1337}+53\biggr) z^2\biggr)
    \end{align*}
    }
    has the desired property (see Fig.~\ref{fig:stubborn_on_circle_and_tangent}). 
    \begin{figure}
        \includegraphics[width=8cm]{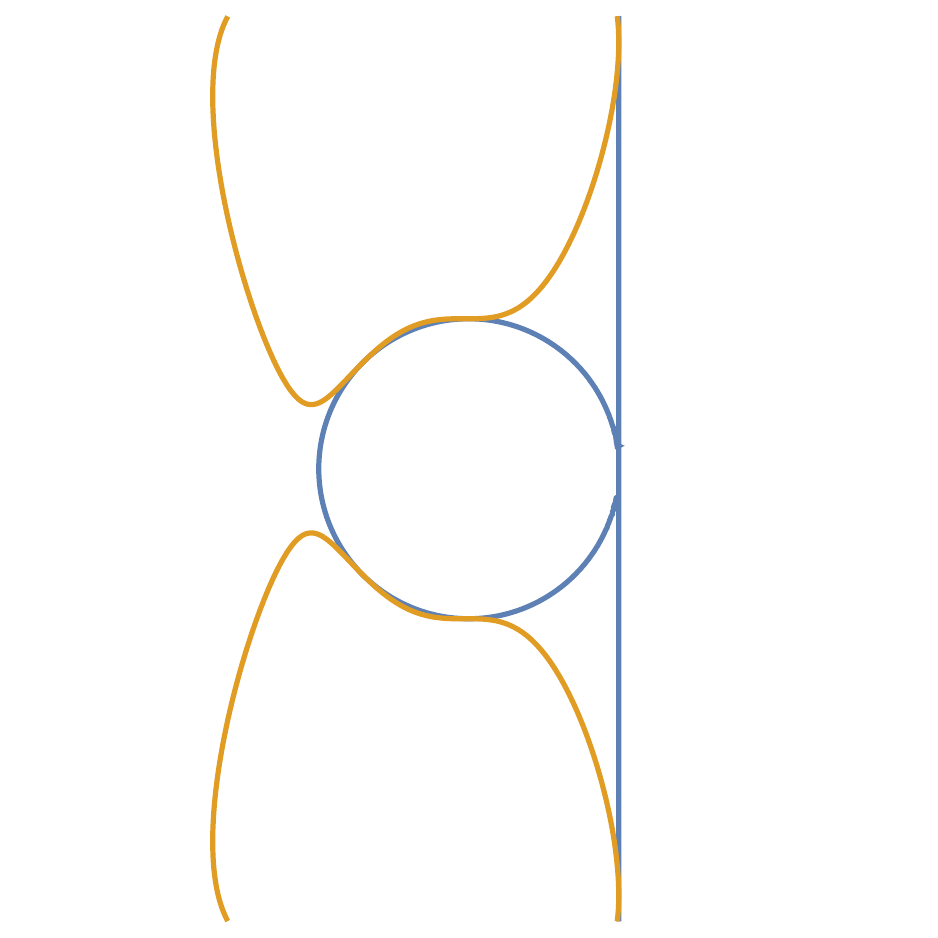}
        \includegraphics[width=8cm]{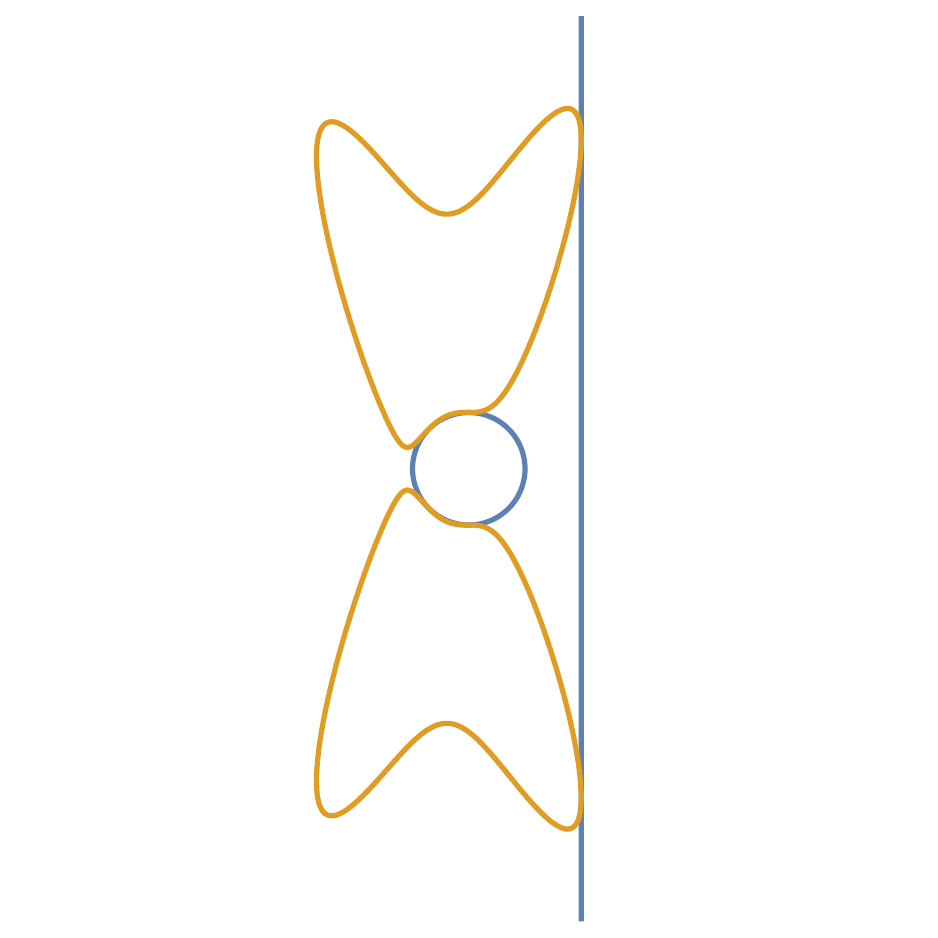}
    \caption{A stubborn quartic for a circle with a tangent or passing line.}\label{fig:stubborn_on_circle_and_tangent}
    \end{figure}

    (c) If $C$ and $L$ have no real intersection point, then $X=C\cup L$ admits no stubborn quadratic forms. This is easier to see than in the tangent case, since the only extremal nonnegative quadratic forms with only real zeros on $X$ are squares: The product of two tangent lines to $C$ has opposite signs on $C$ and $L$, while an irreducible, indefinite quadratic $F$ cannot be nonnegative on $C$ and $L$ while being tangent to both. (One way to see this is to consider the affine plane $\P^2\setminus L$ in which $C$ is an ellipse while the conic $P$ defined by $F$ is tangent to $L$, and therefore a parabola. If $P$ is tangent to $C$ at two points, then $C$ must be contained in the closure of the convex component of $\P^2(\R)\setminus P(\R)$, while $L$ is contained in the other, so that $F$ changes sign on $X$.) 

    But $X$ again admits stubborn forms of degree $4$ of the same type as in (b) (see Fig.~\ref{fig:stubborn_on_circle_and_tangent}). To see this, we may assume after a change of coordinates that $Q=x^2+y^2-z^2$ and $\ell=2z-x$. (To see that such a change of coordinates is indeed possible, observe that the group of projective transformations that preserves the ellipse $C$ acts transitively on points where $Q$ is negative; see for example \cite[Thm.~1.1.6]{Ratcliffe19}) A similar symbolic computation as in case (b) will produce a stubborn quartic, e.g.\\[-2em]

    {\footnotesize
    \begin{align*}
    F = &\biggl(18 \sqrt{2}-5 \sqrt{3516 \sqrt{2}+6539}-131\biggr) x^4-4 \biggl(57 \sqrt{2}+\sqrt{7032 \sqrt{2}+13078}+18\biggr) x^3 z\\ &+2 x^2 \biggl(\biggl(54 \sqrt{2}-\sqrt{3516 \sqrt{2}+6539}+97\biggr) y^2-6 \biggl(4 \sqrt{2}+9\biggr) z^2\biggr)+8 \bigl(4 \sqrt{2}+9\bigr) x z \bigl(y^2-z^2\bigr)\\ &+\bigl(y^2-z^2\bigr) \biggl(\biggl(18 \sqrt{2}-\sqrt{3516 \sqrt{2}+6539}+65\biggr) y^2+\biggl(78 \sqrt{2}+\sqrt{3516 \sqrt{2}+6539}+151\biggr) z^2\biggr)
    \end{align*}
    }
\end{ex}

\begin{ex}
    If $X$ consists of three distinct lines, it is defined by $\ell_1\ell_2\ell_3=0$ for three linear forms $\ell_1,\ell_2,\ell_3$.
    If the lines meet in three distinct points, then $X$ admits a stubborn quadratic form, namely an ellipse in the bounded region of $\P^2(\R)\setminus X(\R)$ tangent to all three lines.

    If the three lines pass through a single point, then $X$ does not admit a stubborn quadratic form. After a linear change of coordinates, we may assume $\ell_1=x$, $\ell_2=y$, $\ell_3=x-y$. 
    It again suffices to check that no extremal nonnegative quadratic form is stubborn. 
    An irreducible conic cannot be tangent to all three lines, so it cannot be stubborn by \cite[Thm.~1.1]{KummerZalar2025}. The other case we need to consider is a product of two lines through the intersection point $[0,0,1]$ of $\ell_1,\ell_2,\ell_3$. One such extremal form is the product of two of the lines. For example, $xy$ is nonnegative on $X$ and not a sum of squares, but $(xy)^3 = x^4y^2 - x^3y^2(x-y)=x^4y^2-x^2y\cdot (\ell_1\ell_2\ell_3)$ shows that $xy$ is not stubborn on $X$. 

    However, $X$ does again admit a stubborn form $F$ of degree $4$. Namely, $F$ can be chosen such that $\ell_1,\ell_2,\ell_3$ are bitangents of the quartic and $F$ has constant sign on $X(\R)$. Explicitly, it can be verified through symbolic computation that
    \[
    F = \det\left(
    \begin{array}{cccc}
    0 & x & y & 4 x+y-z \\
    x & 0 & x+4 y-z & x-y \\
    y & x+4 y-z & 0 & x-y+4 z \\
    4 x+y-z & x-y & x-y+4 z & 0 \\
    \end{array}
    \right)
    \]
    has the desired property (see Fig.~\ref{fig:stubborn_on_three_lines}). For this, it suffices to check that $F$ is irreducible and meets $\ell_1$, $\ell_2$ and $\ell_3$ each in a pair of real double points, since an irreducible quartic will have the same constant sign on all its bitangent lines.
    \begin{figure}
        \includegraphics[width=8cm]{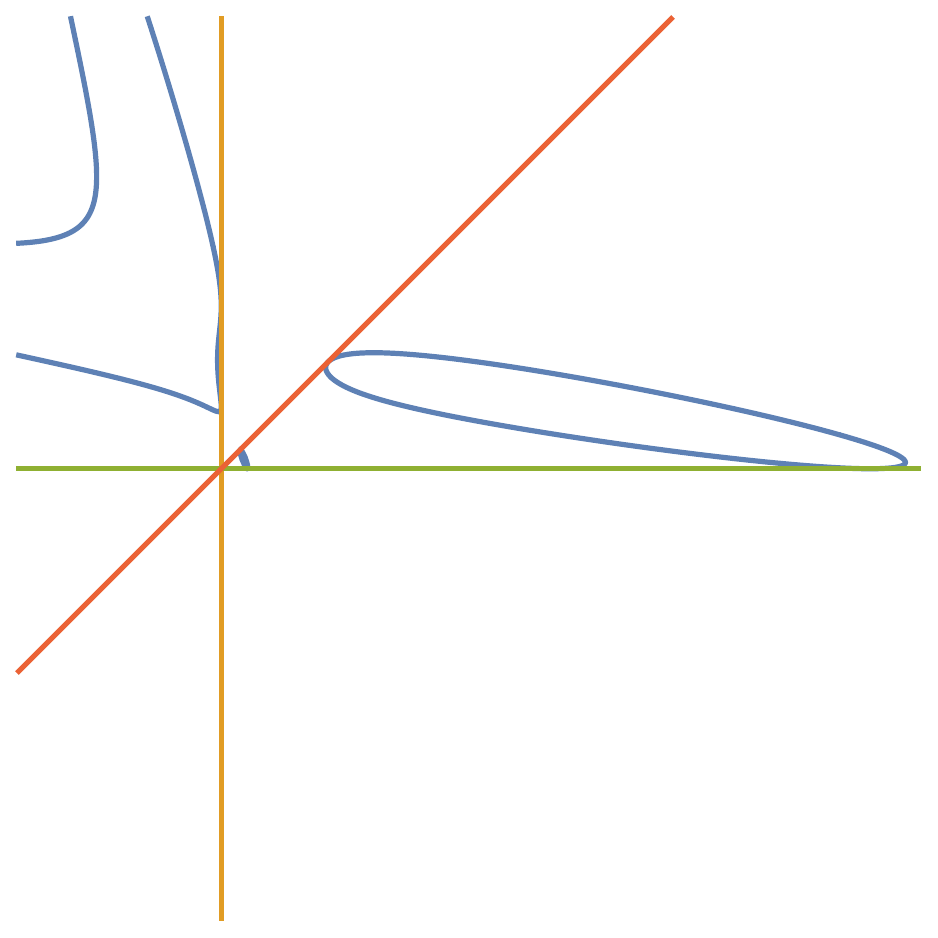}
    \caption{A stubborn quartic on three lines through a point.}\label{fig:stubborn_on_three_lines}
    \end{figure}
\end{ex}

\medskip
The preceding series of examples covers all totally real plane cubics. We sum up our findings as follows.

\begin{theorem}\label{thm:planecubics}
Let $X\subset\P^2$ be a totally real cubic curve. 
\begin{enumerate}
    \item $X$ admits a stubborn quadratic form in the following cases:
    \begin{enumerate}[label = {(\theenumi.\arabic*)}]
        \item $X$ is a smooth cubic;
        \item $X$ is a nodal cubic with $X(\R)$ connected;
        \item $X$ is a smooth conic with a secant line;
        \item $X$ consists of three lines in general position.
    \end{enumerate}
    \item $X$ admits no stubborn quadratic form, but does admit a stubborn quartic form in the following cases:
    \begin{enumerate}[label = {(\theenumi.\arabic*)}]
        \item $X$ is a smooth conic with a tangent line;
        \item $X$ is a smooth conic with a passing line;
        \item $X$ consists of three distinct lines through a single point.
    \end{enumerate}
    \item $X$ admits no stubborn forms of any degree in the following cases:
    \begin{enumerate}[label = {(\theenumi.\arabic*)}]
        \item $X$ is a cuspidal cubic;
        \item $X$ is a nodal cubic with a solitary node;
        \item $X$ consists of one or two lines (with repeated factors).\qed
    \end{enumerate}
\end{enumerate}
\end{theorem}

\subsection{Singular curves and positive $2$-torsion points}
\label{sec:2torsion}
In the previous section, we discussed the existence of stubborn forms on singular and reducible curves, mostly using explicit computations, possible only because we considered plane cubics. For smooth curves, we showed more conceptually that stubbornness is equivalent to real-rootedness, and we used this in \Cref{thm:large degree} to prove that stubborn polynomials of sufficiently high degree exist on all smooth curves of positive genus. We now discuss the necessary ingredients for possible extensions of this result to the singular case.

In \cite{baldiblekhermannsinn24}, the existence of nonnegative polynomials with the maximal number of real zeros is deduced from the existence of \emph{positive} $2$-torsion points in the Jacobian, see \cite[Def.~3.4.4]{baldiblekhermannsinn24} and \cite[Cor.~2.3]{KummerZalar2025}. This definition can be extended to singular curves, by studying $2$-torsion points of the generalized Jacobian instead. For smooth curves we know that there are $2^g$ positive $2$-torsion points, but for singular curves, computing their number is still an open question. However, if $X(\R)$ is connected then every polynomial with zeroes on $X(\R)$ of even multiplicity does not change sign on $X(\R)$, and thus \emph{every} {real} $2$-torsion point of the generalized Jacobian is a positive $2$-torsion point.

We can also connect $2$-torsion points with effective totally real divisors on $X$ by \cite[Thm.~B]{monnier04}. In particular, this result applies if the singularities of $X$ are all real.
Combining these two ingredients, we obtain the following result.

\begin{Thm}
\label{thm:daniel_bis_monnier}
Let $X\subset\P^n$ be an irreducible real curve. Assume that $X(\R)$ is connected and that all singularities of $X$ are real. If the generalized Jacobian $J_X$ has a non-trivial $2$-torsion point, then $X$ admits stubborn nonnegative polynomials (of some even degree) in all sufficiently high degrees.
\end{Thm}
\begin{proof}
    Let $\eta\in J_X$ be a non-trivial $2$-torsion point. Choose a real hyperplane that intersects $X$ in smooth points only, not necessarily real, and let $H$ be the intersection divisor. Since $X$ has no isolated or non-real singularities, \cite[Th.~B]{monnier04} implies that the class $\eta+d\cdot[H]$ contains an effective divisor $D$ supported on $X_{\mathrm{reg}}(\R)$ for sufficiently large $d$. Moreover, for all sufficiently large $d$, $X$ is $(2d\deg X)$-normal (see \Cref{def:dnormal}): since $2D$ is linearly equivalent to $2dH$, this implies that there exists a form $F\in\R[X]_{2d}$ with $\div(F)= \div(F)_\R =2D$. 
    
    Now, as $X(\R)$ is connected and all zeros of $F$ on $X(\R)$ are regular points with even multiplicity, $F$ has constant sign on $X(\R)$. Replacing $F$ by $-F$ if needed, we may assume that $F$ is nonnegative. It follows from \Cref{thm:many_zeroes_implies_stubborn} that $F$ is stubborn.
\end{proof}

\section{Curves with no stubborn polynomials}\label{sec:nostub}

As seen in Section \ref{sec:cubics}, stubborn polynomials on a curve may fail to exist, and we give several examples and constructions of this behavior. As we discuss below, on smooth rational curves stubborn polynomials do not exist in any degree, while on smooth curves of positive genus, stubborn polynomials may fail to exist only in low degree (see Theorem \ref{thm:large degree}).

\subsection{Smooth rational curves}\label{sec:smoothnostub}

Let $X\subset \mathbb{P}^n$ be a smooth totally real curve of genus $0$. 
By considering the Hilbert polynomial of $X$, we see that for sufficiently large $k$, the degree $k$ part $\R[X]_k$ of the coordinate ring of $X$ is the same as that of the rational normal curve. On the rational normal curve, nonnegative polynomials are always sums of squares, and therefore, there are no stubborn polynomials on $X$. As an explicit example, we can let $X$ be any  monomial curve whose parametrizing monomials include $s^d,s^{d-1}t, st^{d-1}$ and $t^d$.

\subsection{Smooth plane curves with no real-rooted quadrics}
Given nonnegative integers $d$, $n$ and $r$ with $d,n$ even, $r<d$, and $r\leq n$, we construct a real plane curve of degree $d+n$ that intersects any degree $r$ curve in at most $dr$ real points counted with multiplicities.
We owe the idea of the proof to Orevkov \cite{Orevkov}.

Let $X\subset \P^2$ be a smooth real projective plane curve defined by $F\in \R[x,y,z]_d$ and such that $X(\R)$ is non-empty and connected.
Let $Q, R\in \R[x,y,z]$ be positive definite forms of degrees $n$ and $d+n$ respectively, such that the curve defined by $R$ is smooth, and $Q$ is irreducible.
We consider a family of real curves $X_t\subset \P^2$ defined by forms $F\cdot Q+ t R$ of degree $d+n$.
Note that $X_t(\R)$ converges to $X(\R)$ in the Gromov-Hausdorff 
sense.
Also, for all sufficiently small $t>0$, the curve $X_t$ is smooth, and therefore totally real.

\begin{theorem}
\label{thm:deformation}
Let $r\geq 1$. For all sufficiently small $t>0$, the smooth real curve $X_t$ of degree $d+n$ defined above intersects any curve $C\subset \P^2$ of degree $r$ in at most $dr$ real points (counted with multiplicity).
In particular, $X_t$ does not admit stubborn forms of degree $r$.
\end{theorem}

\begin{proof}
Assume that there is a sequence of $(t_k)_{k\in\N}$ converging to $0$ and a sequence of real curves $C_{t_k}$ of degree $r$ intersecting $X_{t_k}$ in at least $dr+1$ real points counted with multiplicities.
By passing to a subsequence, if necessary, we can assume that $C_{t_k}$ converges to some real curve $C$ of degree $r$.
The curve $X_0$ is reducible with one irreducible component being $X$, while the other component defined by $Q$ has an empty real locus.
In particular, $X_0$ intersects the curve $C$ in at least $nr=\deg Q \cdot \deg C$ nonreal points (counted with multiplicities).
It follows from \cite[p. 664]{GH1994} that the (total) intersection number of $X_0$ and $C$ in a small analytic neighborhood $U\subset \P^2\setminus \P^2(\R)$ of $a\in X_0\cap C$ is invariant under continuous deformations of $X_0$ and $C$.
Hence, for all large $k$ there must be at least $nr$ nonreal points in the intersection of $X_{t_k}$ and $C_{t_k}$, which contradicts our assumption that we have at least $dr+1$ real intersection points.
The second claim follows from \Cref{thm:stubborn_implies_many_zeroes}.
\end{proof}

Taking $r=2$ and any even $d\geq 4$ and $n\geq 2$, we obtain the following corollary.

\begin{cor}
    For all sufficiently small $t>0$ the smooth real curve $X_t$ of degree $d+n$ intersects any conic $C\subset \P^2$ in at most $2d$ real points (counted with multiplicities).
   In particular, $X_t$ does not admit stubborn quadrics.

\end{cor}

\begin{rem}
    Denote by $|rH_t|$ the linear system cut out by curves of degree $r$ in $\P^2$ on $X_t \subset \P^2$. In the previous theorem, we proved the following: for all $t$ small enough, if $D \in |rH_t|$ then $\#(\mathrm{supp}(D) \cap X_t(\R))\le dr$.
    {The \emph{totally real divisor threshold} of $X_t$, denoted $\mathrm{N}(X_t)$, is the smallest integer such that every complete linear system of degree at least $\mathrm{N}(X_t)$ contains a totally real divisor. The totally real divisor threshold cannot be bounded solely in terms of the topology of the real curves, but rather depends on metric properties of the curve. Estimating it is nontrivial. \Cref{thm:deformation} then shows that $\mathrm{N}(X_t) > (d+n)r$ for all $t$ small enough.}
    We refer the reader to \cite{baldi2025totallyrealdivisorscurves} for more details on the totally real divisor threshold.
\end{rem}

\subsection{No real-rooted polynomials via Descartes' rule of signs.}
Another possible way of ensuring nonexistence of real-rooted conics is via Descartes' rule of signs, see \cite{Descartes} or \cite[pp. 96-99]{Struik}.
Let $X$ be the curve in $\mathbb{P}^2$ with parametrization given, on an affine chart, by $t\in \mathbb{A}^1\mapsto (1+t^2+t^4, t^{9}(1+t^2))\in \mathbb{A}^2$.
In coordinates $(x,y)$ in the affine plane, this rational curve of degree $11$ is defined by
\begin{equation}
\begin{aligned}
f\ =\ & x^{11} - 11 x^{10} + 55 x^9 - 165 x^8 + 330 x^7 - 7 x^5 y^2 - 462 x^6 + 21 x^4 y^2 + 462 x^5 \\ &- 21 x^3 y^2 - 330 x^4 + 6 x^2 y^2 - y^4 + 165 x^3 + 2 x y^2 - 55 x^2 - y^2 + 11 x - 1.
\end{aligned}
\end{equation}

The following result shows that no real conic can intersect the curve $X$ only in real points.

\begin{prop}
Let $C\subset \P^2$ be a real conic. Then $C$ intersects $X$ in less than $22=2\deg X$ real points counted with multiplicities. In particular, $\#(X(\R)\cap C(\R))<2\deg X=22$.
\end{prop}

\begin{proof}
Pulling back a (real) quadric $a_{11}x^2+a_{12}xy+ a_{22} y^2 + a_1 x + a_2 y+ a_0$ via the parametrization of $X$ gives a univariate (real) polynomial $q$ of degree (at most) $22$ whose monomial support is limited to the degrees $0, 2, 4, 6, 8$, $9, 11, 13, 15$ and $18, 20, 22$.
The monomials of degrees $0, 2, 4, 6, 8$ come from $a_{11}x^2, a_1 x$, and the constant term $a_0$; the monomials of degrees $9, 11, 13, 15$ come from $a_{12} xy$ and $a_2 y$;
and the monomials of degrees $18, 20, 22$ come only from $a_{22}y^2$.

Let us first consider the case of a reducible (over $\R$) conic and show that it cannot have only real intersection points with $X$.
For this it is enough to observe that no line $a_1x + a_2 y + a_0$ intersects $X$ in $11$ real points counted with multiplicities.
This directly follows from Descartes rule of signs, see e.g \cite[Th.~1.3.12]{Scheiderer2024}, as $(a_0+a_1)+ a_1t^2+a_1t^4+a_2 t^9+a_2 t^{11}$ has too few sign variations in the sequence of its coefficients.

Therefore, we can assume that $a_{11}x^2+a_{12}xy+ a_{22} y^2 + a_1 x + a_2 y+ a_0$ is irreducible.

If $a_{22}\neq 0$, the polynomial $q$ is of degree $22$ but it has at most $12$ nonzero terms.
Moreover, since terms of degrees $18, 20$ and $22$ have coefficients of the same sign, there are at most $9$ sign variations in the sequence of coefficients of $q$.
In general, one can factor $q=t^k \tilde q$, where $k\leq 18$ and $\tilde q(0)\neq 0$.
The polynomial $\tilde q$ is of degree $22-k$ and, depending on the value of $k=0,2,4,6,8, 9, 11, 13, 15, 18$, the sequence of  coefficients of $\tilde q$ has at most $9, 8, 7, 6, 5, 4, 3, 2, 1,$ or $0$, respectively, sign variations.
In any of these cases Descartes rule of signs implies that $q=t^k\tilde q$ has at most $18<22=\deg q$ real zeros counted with multiplicities.

If $a_{22}=0$ but $a_{12}\neq 0$, the degree of $q$ is $15$ but it has at most $9$ nonzero terms.
Moreover, since terms of degrees $13$ and $15$ have the same coefficient $a_{12}$, there are at most $7$ sign variations in the sequence of coefficients of $q$.
Again, factoring $q=t^k \tilde q$ with $k\leq 13$ and $\tilde q(0)\neq 0$ we see that
the polynomial $\tilde q$ has degree $15-k$ and, depending on the value of $k=0,2,4,6,8,9,11,13$, the sequence of coefficients of $\tilde q$ has at most $7, 6, 5, 4, 3, 2, 1$ or, respectively, $0$ sign variations.
In all of these cases Descartes rule of signs guarantees that $q=t^k\tilde q$ has no more than $14<15=\deg q$ real zeros counted with multiplicities.

Let now both $a_{22}$ and $a_{12}$ be zero.
Then $a_{11}\neq 0$, as otherwise we are in the reducible case.

If $a_2\neq 0$ (given that $a_{22}=a_{12}=0$), the polynomial $q$ is of degree $11$ and it has at most $7$ nonzero terms.
Also, terms of degrees $9$ and $11$ have the same coefficient $a_2$ and so there are at most $5$ sign variations in the sequence of coefficients of $q$.
By the same argument as above, given the degree $11$ polynomial $q=t^k \tilde q$, $k\leq 8$, $\tilde q(0)\neq 0$, there are at most $5, 4, 3, 2$ or $1$ sign variations in the sequence of the coefficients of $\tilde q$ depending on the respective value of $k=0, 2, 4, 6, 8$.
And by Descartes rule of signs, $q=t^k\tilde q$ cannot have more than $10<11=\deg q$ real zeros counted with multiplicities.

Finally, if $a_{22}=a_{12}=0$ and also $a_2=0$, the polynomial $q$ has degree $8$ and it has at most $5$ nonzero terms.
Moreover, as terms of degrees $6$ and $8$ have coefficients of the same sign, there are no more than $3$ sign variations in the sequence of coefficients of $q$.
We can again write $q=t^k \tilde q$ with $k\leq 4$ and $\tilde q(0)\neq 0$.
There are at most $3, 2$ or $1$ sign variations in the sequence of coefficients of $q$, if $k=0, 2$ or, respectively, $4$.
By Descartes rule of signs, $q=t^k\tilde q$ has no more than $6<8=\deg q$ real zeros counted with multiplicities.
This concludes the proof.
\end{proof}

\begin{rem}
With the same proof as in Theorem \ref{thm:deformation} we can also perturb the rational curve $X$ constructed above to obtain a smooth curve $X'$ so that any real conic does not have only real intersection points with $X'$. 
\end{rem}

\section{Nonnegative lifting}\label{sec:lifting}

We now discuss how to extend nonnegative quadrics on a curve to nonnegative quadrics on a larger variety containing the curve. In view of the result from \cite{MR3486176}, this is in general not possible (for example, if the variety is not of minimal degree).
However, under some assumptions on the curve and the ambient variety, we provide a method for such a \emph{nonnegative lifting}.

One of our goals is to build stubborn polynomials on larger dimensional varieties, whose stubbornness is inherited from stubbornness on a curve. This is based on the following simple observation. Let $X \subset Y$ be a subvariety and denote by $\iota^* \colon \R[Y]\to \R[X]$ the associated pullback map between homogeneous coordinate rings. 
\begin{prop}
    \label{prop:inherit_stubbornness}
    If $X\subset Y$ is a closed totally real variety in a totally real variety $Y$ and $f\in \R[X]$ is stubborn, then any nonnegative $g\in \R[Y]$ with $\iota^*(g) = f$ is stubborn on $Y$.
\end{prop}

Following e.g.~\cite{Miranda1995}, we say that an effective divisor $E=\sum m_i p_i$ on a curve $C$ \emph{imposes independent conditions} on forms of degree $k$ on $C$ if the codimension of the vector subspace of $\R[C]_k$ consisting of forms $F$ such that $\div(F)-E$ is effective, is equal to the degree of $E$. 
\begin{rem}\label{rem:subdivisor}
Observe that if $E=\sum m_i p_i$, $m_i\geq 0$, imposes independent conditions on forms of degree $k$ on $C$, then so does any divisor $E'=\sum m'_i p_i$ with $0\leq m'_i\leq m_i$.
\end{rem}

\begin{theorem}\label{thm:general_lifting}
Let $X\subseteq \P^n$ be a totally real variety with $I(X)$ defined by quadrics. Let $L\subset \P^n$ be a linear space, such that $C=L\cap X$ is a smooth curve and for all $p\in C$ we have $$T_p C=T_p X\cap L.$$
Let $F\in \R[C]_2$ be a nonnegative quadratic form on $C$ and let $D=2E=\div(F)_\R$ be the divisor of real zeroes of $F$ on $C$. Suppose furthermore that points in $E$ impose linearly independent conditions on linear forms on $C$. Then there exists $G\in I(L)_2$ such that $F+G$ is nonnegative on $X$.
\end{theorem}

\begin{rem}
The containment $T_p C\subseteq T_p X\cap L$ always holds. We require above that there is no additional intersection between $T_p X$ and $L$. Notice that the condition $T_p C=T_p X\cap L$ does not imply that $X$ and $L$ intersect transversely at $C$, since $T_p X$ and $L$ are not required to span $T_p {\P^n}$.
\end{rem}

\begin{proof}
First, we do the local lifting; that is, we extend $F$ to a form on $X$ that is nonnegative in some analytic neighborhood in $X(\R)$ of zeroes of $F$ on $C(\R)$. Let $p$ be a real zero of $F$ on $C$. The assumption on the tangent spaces is equivalent to $T_p X^\perp+L^\perp=T_p C^\perp$. 
Then we can pick linear forms $t, \ell_1,\dots, \ell_{N-1}$ (where $N=\dim X$) that give rise to a basis of the cotangent space of $X$ at $p$ and such that $\ell_1,\dots,\ell_{N-1}\in I(L)_1$.
Then \cite[Prop. 3.3.7]{BochnakCosteRoy1998} implies that $t,\ell_1,\dots,\ell_{N-1}$ form a regular system of (local) parameters of $X$ at $p$ with $t$ being a local parameter of $C$ at $p$.

Expanding $F$ locally on $C$ at $p$, we write $F=c_{2m}t^{2m}+ O(t^{2m+1})$. 
Since $F$ is nonnegative on $C$, the multiplicity $2m$ of the intersection of $F$ with $C$ must be even and $c_{2m}>0$. Now let $G=\ell_1^2+\dots+\ell_{N-1}^2$ and consider the local expansion of $F+\lambda G$ on $X$ at $p$ for a sufficiently large $\lambda>0$. 

\textbf{Claim.} The form $F+\lambda G$ is locally around $p$ nonnegative on $X(\R)$ for sufficiently large $\lambda>0$ if and only if the local expansion of $F$ does not contain the terms $\ell_i, \ell_i t,\dots,\ell_i t^{m-1}$ for all $i=1,\dots, N-1$. 

\emph{Proof of Claim.} If any of these terms are present, then a leading term in $F+\lambda G$ is not even, and thus $F+\lambda G$ is not locally nonnegative. For the other direction, the principal part of $F+\lambda G$ is a quadratic form in $\ell_1,\dots,\ell_{N-1}$ and $t^{m}$ involving $c_{2m} t^{2m} + \lambda(\ell_1^2+\dots+\ell^2_{N-1})$. Note that even if this quadratic form contains terms $\ell_i t^m$, for sufficiently large $\lambda>0$ it takes strictly positive values away from $t=\ell_1=\dots=\ell_{N-1}=0$. This follows, for example, by Schur complement criterion for positive definiteness of a real symmetric matrix \cite[Thm. 7.7.7]{HornJohnson}. 
Finally, by \cite[Prop. 1.3]{Vasiliev}, in a small analytic neighborhood of $t=\ell_1=\dots=\ell_{N-1}=0$ the sign of $F+\lambda G$ is determined by the sign of its principal part, which (as shown) is nonnegative with $p$ being a strict local minimum.

Therefore, our goal is to cancel terms $\ell_i, \ell_i t,\dots,\ell_i t^{m-1}$ in a local (on $X$ at $p$) expansion of $F$ by adding forms $s_i\ell_i\in I(L)_2$.
Writing $F$ locally at $p$, we obtain
\begin{align}\label{eq:expansion}
F = P(t)+\ell_1Q_1(t)+\dots+\ell_{N-1}Q_{N-1}(t)+ \text{ mixed terms}
\end{align}
We want to find linear forms $s_1,\dots, s_{N-1}\in \R[X]_1$ having the same expansion as $-Q_1(t),\dots, -Q_{N-1}$ up to degree $m-1$ respectively. Then, by the above claim, $F+\sum s_i\ell_i+\lambda G$ is locally nonnegative for sufficiently large $\lambda$.
We explain how to find such a linear form $s\in \R[X]_1$ for a given $R(t):=-Q(t)$.
The condition that terms $1, t,\dots, t^{m-1}$ do not appear in $s+Q$ just means that it intersects $C$ with multiplicity at least $m$. 
Finding $s$ thus means solving a system of \emph{affine} linear equations. 
Indeed, let us write $s=\sum_{j=0}^n a_j x_j$.
Moreover, the coordinates restricted to the curve $C$ admit expansions $x_j|_C=\sum x_{jk} t^k$, $j=0,1,\dots, n$.
The condition that $s|_C$ agrees with $R(t)=-\sum_{k\geq 0} q_k t^k$ up to order $k\leq m-1$ at $p\in C$ is then written explicitly as $\sum_{j=0}^n x_{jk} a_{j} = -q_k$, $k=0,1,\dots,m-1$.
This is an affine system of size $m\times (n+1)$.
The corresponding matrix $(x_{jk})_{j,k=0}^{n,m-1}$ is of rank $m$, since otherwise the divisor $E'=m p$ (and hence also the whole divisor $E$) does not impose independent conditions on linear forms on $C$ (cf. Remark \ref{rem:subdivisor}).
This ensures that we can find $a_0, a_1, \dots, a_n\in \R$ solving the above affine system, or equivalently, that $s|_C=\sum_{j=0}^n a_j x_j|_C$ agrees with $R(t)=-Q(t)$ up to order $m-1$.  
Let now $p_1,\dots, p_r$ be points in the support of $E$ whose multiplicities are $m_1,\dots, m_r$, that is, $E=\sum m_i p_i$. Since $E$ imposes independent conditions on linear forms, the matrix (as above) encoding that $s|_C$ agrees with any given $R_i(t)$ up to order $m_i-1$, $i=1,\dots, r$, is of rank $m_1+\dots+m_r$. This ensures that we can find $s\in \R[X]_1$ satisfying the desired conditions at all points in the support of $E$. 
Summarizing, by adding to $F$ the quadratic form $s_1\ell_1+\dots+s_{N-1} \ell_{N-1}+\lambda(\ell_1^2+\dots+\ell_{N-1}^2)$, we can make it locally nonnegative at all $p_1,\dots, p_r$. 

Now we can apply the lifting argument by adding a sufficiently large positive multiple of $G$ to lift $F$ to a quadratic form nonnegative on $X$. 
Indeed, to see this explicitly, let us pass to the double cover $\mathbb{S}^n$ of $\mathbb{P}^n({\mathbb{R}})$ and denote  the preimages of $X$ and $p_1,\dots, p_r$ by $\tilde X\subset \mathbb{S}^n$ and $\pm \tilde p_1,\dots,\pm \tilde p_r\in \tilde X$, respectively. 
Outside small neighborhoods of $\tilde p_1,\dots, \tilde p_r$ in $\tilde X$ the quadratic polynomial $F$ attains its minimum $F_{\min}$. If $F_{\min}\geq 0$, we have nothing to add to make $F$ nonnegative on $X$.
Otherwise, assume we have already made $F$ nonnegative in small neighborhoods of $p_1,\dots, p_r$. 

Then $F_{\min}<0$ is attained outside a tubular neighborhood $U$ of $\tilde C$ in $\tilde X$.
But over $\tilde X\setminus U$ the form $G$ is bounded from below by some $G_{\min}>0$.
Therefore, for possibly even larger $\lambda'>0$ the form $F+\lambda' G$ is nonnegative on $\tilde X$ and hence on $X$.
\end{proof}

\begin{rem}
    Let $X$ be a real variety, $C\subset X$ be a curve, and suppose that the ideal of $C$ inside the ideal of $X$ is defined by quadrics. Then $\nu_2(C)=\nu_2(X)\cap L$ where $L$ is a linear space and for all $p\in \nu_2(C)$ we have $$T_p {\nu_2(C)}=T_p {\nu_2(X)}\cap L.$$
\end{rem}

\begin{cor}\label{cor:ext}
Let $X\subseteq \P^n$ be a totally real variety with $I(X)$ defined by quadrics. Let $L\subset \P^n$ be a linear space, such that $C=L\cap X$ is an elliptic normal curve in $L$, and $T_p C=T_p X\cap L.$
Then for any quadric $F\in \R[C]_2$ nonnegative on $C$, there is a quadratic form $\hat{F}\in \R[X]_2$ nonegative on $X$ such that $\hat{F} \mid_C=F$, i.e.,~$F$ can be extended to a quadratic form nonnegative on $X$.
\end{cor}
\begin{proof}
We need to verify that $\frac{1}{2}\div(F)_\R$ imposes independent conditions on linear forms on an elliptic normal curve $C$. Assume first $F\in \R[C]_2$ spans an extreme ray of the cone of nonnegative quadratic forms on $C$. Then by \cite[Thm. B]{baldiblekhermannsinn24} we have $\div(F)=\div(F)_\R$. If $\frac{1}{2}\div(F)$ imposed dependent conditions on linear forms $\R[C]_1$, there would exist a linear form $\ell\in \R[C]_1$ that vanishes on $\frac{1}{2}\div(F)$. Then $F-\varepsilon \ell^2$ and $F+\varepsilon \ell^2$ are nonnegative quadratic forms not proportional to $F$, where $\varepsilon>0$ is sufficiently small. This contradicts the extremality of $F=\frac{1}{2}(F-\varepsilon \ell^2)+\frac{1}{2}(F+\varepsilon \ell^2)$.
So, by Theorem \ref{thm:general_lifting}, $F$ can be extended to a quadratic form nonnegative on $X$.
If $F=\sum_{i=1}^r F_i$ is not extremal, it can be extended to a quadratic form nonnegative on $X$, since this holds for extremal quadratic forms $F_i\in \R[C]_2$ that generate $F$.
\end{proof}
We can formulate Corollary \ref{cor:ext} as a statement about projection of the cones of nonnegative polynomials on real varieties, as follows.
Let $X\subseteq \P^n$ be a totally real variety whose vanishing ideal $I(X)$ is generated by quadratic forms. Let $L\subset \P^n$ be a linear space, such that $C=L\cap X$ is an elliptic normal curve in $L$, and $T_p C=T_p X\cap L.$
Let $I(C)$ be the ideal of $C$ inside $\R[X]_2$. Then we have a natural linear projection $\pi: \R[X]_2\rightarrow \R[C]_2$, which mods out by $I(C)$. It is clear that $\pi(P_X)\subseteq P_C$ holds in full generality. Under the above assumptions we have the equality $\pi(P_X)= P_C$.
\begin{cor}
    \label{cor:ext_elliptic}
Let $C$ be an elliptic normal curve of degree $n$ in $\P^n$. Then any form of degree $2d$  ($d\geq 3$ for $n=2$, and $d\geq 2$, $n\geq 3$) nonnegative on $C$, can be extended to a globally nonnegative form $f$ on $\P^n$. In particular, there exist stubborn forms of degree $2d$ on $\P^n$ with at most $nd$ real zeroes, where each real zero is a simple node (i.e.,~the Hessian at every real zero of $f$ is positive semidefinite of rank $n$).
\end{cor}

\begin{proof}
We apply Corollary \ref{cor:ext} to the $d$-th Veronese embeddings of $C$ and $\P^n$. The ideal of $C$ in $\P^n$ is defined by quadrics when $n\geq 2$ and by a single cubic when $n=2$. Since we have $d\geq 2$ for $n\geq 3$ and $d\geq 3$ for $n=2$, we see that $\nu_d(C)=L\cap\nu_d(\P^n)$ and the tangent space conditions hold. We can build a stubborn polynomial $F$ of degree $2d$ on $C$ with exactly $nd$ real zeroes. The construction of the lifting in the proof of Theorem \ref{thm:general_lifting} with adding $\ell_1^2+\dots+\ell_{k-1}^2$ ensures that all the real zeroes of the lifted form $\tilde F$ are simple nodes. Also, $\tilde F$  inherits its stubbornness from $F$ by Proposition \ref{prop:inherit_stubbornness}.
\end{proof}

\begin{ex}
We now show that if the points of intersection of $C$ and $F$ do not impose independent conditions on linear forms on $C$, then nonnegative lifting is not always possible.
For this, let us consider the Edge quartic curve $E$ in $\mathbb{P}^2$ studied in \cite{MR2781949}. It is cut out by a quartic polynomial
\begin{align*}
   G = 25(x^4+y^4+z^4)-34(x^2y^2+x^2z^2+y^2z^2).
\end{align*}
To bring this example in line with Theorem \ref{thm:general_lifting}, we let $X\subset \mathbb{P}^{14}$ be the $4$-th Veronese embedding of $\mathbb{P}^2$ and let $C$ be the $4$-th Veronese embedding of the Edge quartic. Then $C$ is a transverse hyperplane section of $X$, and in particular the intersection conditions of Theorem \ref{thm:general_lifting} are satisfied. A quadratic form on $C$ is an octic form on $E$. 

\begin{figure}[h!]
    \includegraphics[width=8cm]{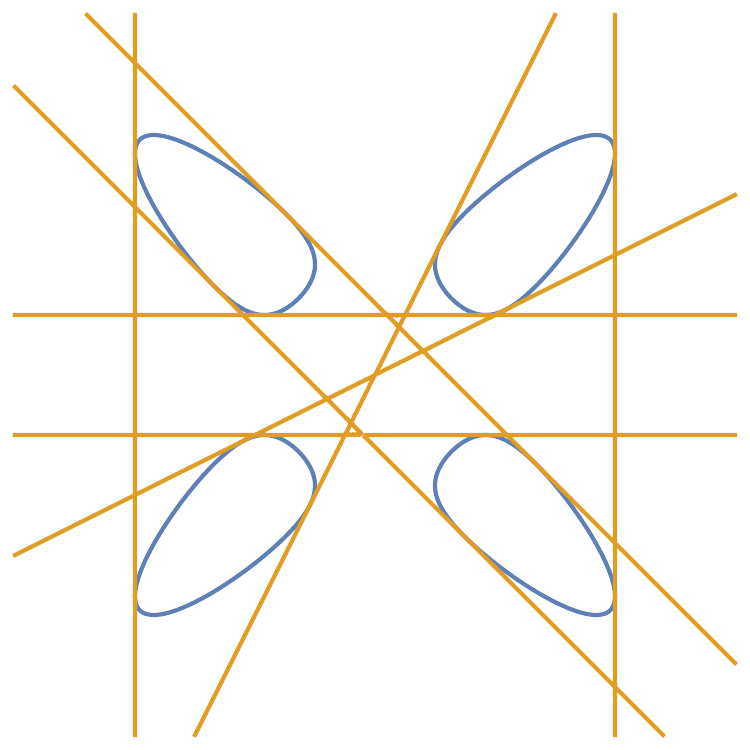}
   
\caption{The Edge quartic and nonnegative octic built from its bitangents.}\label{fig:Edge}
\end{figure}

To draw pictures (see Figure \ref{fig:Edge}), we work with dehomogenized polynomials, and so let $g(x,y)=25(x^4 + y^4 + 1) - 34 (x^2y^2 + x^2 + y^2)$. Let $f$ be the product of the following $8$ bitangents of $E$: $b_1=x-2$, $b_2=-x-2$, $b_3=\frac{1}{2}-y$, $b_4=-\frac{1}{2}-y$, $b_5=x-2y$, $b_6=x-\frac{1}{2}y$, $b_7=x+y+\frac{3}{5}$, $b_8=x+y-\frac{3}{5}$. Then $f=b_1\cdots b_8$ is nonnegative and real-rooted on $C$. 

The divisor $D$ of the homogenization $F$ of $f$ on $C$ consists of $16$ points each of multiplicity $2$, that is, $D=2\cdot \left( p_1+\dots+p_{16} \right)$. Therefore, $E=\frac{1}{2}D=p_1+\dots+p_{16}$ consists of $16$ points of multiplicity $1$. Since $C\subset \mathbb{P}^{13}$, it is not possible for the points of $E$ to impose independent conditions on linear forms on $C$.  

In order to lift $F$ to be nonnegative on $\mathbb{P}^2$, there needs to exist a quartic multiplier $H$ such that the gradient of $F-G\cdot H$ is 0 at the 16 intersection points $p_1,\dots, p_{16}$ of $F$ and $G$. In other words, we must have $H(p_i)=\frac{||\nabla F (p_i)||}{||\nabla G (p_i)||}$ for $i=1,\dots,16$. Here, we are then solving a system of $16$ affine linear equations in $14$ variables, and so we do not expect a solution. A direct computation shows that indeed such a solution does not exist.  
\end{ex}

\section{Ternary sextics and weak Del Pezzo surfaces}

We now consider nonnegative forms of degree $6$ in $\P^2$. In analogy with curve results, we show that stubbornness of a ternary sextic $F$ is determined by the number of real zeroes of $F$, counted appropriately via the real delta invariant defined below. 
It is well-known that a reducible nonnegative ternary sextic, or a nonnegative ternary sextic with infinitely many zeroes is a sum of squares \cite{Choi1980RealZO}. We use these facts repeatedly throughout this section. Moreover, it is easy to see that the maximal number of real zeroes, counted via the real delta invariant, for a nonnegative ternary sextic that is not a sum of squares is 10 -- this comes from the degree-genus formula for planar curves. It was shown in \cite{Stubborn2024} that a nonnegative ternary sextic with 10 zeroes is stubborn. We will improve this result by showing a conjecture from \cite{Stubborn2024}, claiming that a nonnegative ternary sextic which is not a sum of squares is stubborn if and only if its real delta invariant is at least $9$.

First, we show that a ternary sextic with at most $8$ zeroes is not stubborn, using a connection with the theory of weak Del Pezzo surfaces.

\subsection{Sextics with few zeroes are not stubborn}
To better illustrate the main idea behind the fact that nonnegative ternary sextics with few real zeroes are not stubborn, let us first discuss the case of zeroes in general position, which is easier on a technical level.
Following the notation in \cite[Surfaces de Del Pezzo III]{demazure}, in the following we write $X(p)$ for the blow-up of a surface $X$ in a point $p\in X$ and $X(\Phi)$ for the blow-up of $X$ in a finite set of points $\Phi \subset X$, or an iterated blow-up $\pi\colon X(p_1,\ldots,p_r) \to X(p_1,\ldots,p_{r-1}) \to \ldots \to X(p_1) \to X$ with $\Phi = (p_1,\ldots,p_r)$ and $p_i \in X(p_1,\ldots,p_{i-1})$. 

\begin{ex}\label{ex:genericSextic}
    Let $F\in \R[x,y,z]_6$ be a nonnegative ternary sextic with $r\leq 8$ real zeroes in $\P^2(\R)$. We assume that these zeroes $p_1,\ldots,p_r$ are ordinary double points (that is, the Hessian of $F$ at $p_i$ has rank $2$) and in general position in the following sense: no three points are collinear and no 6 points are on a conic. We show that $F$ is then not stubborn, by applying \cite[Theorem 4.11]{Scheiderer2011}; the main case to discuss is when $F$ is not a sum of squares. However, \cite[Theorem 4.11]{Scheiderer2011} does not directly apply, since our polynomial $F\in H^0(\P^2,\sO_{\P^2}(6)) = \R[x,y,z]_6$ is not strictly positive. To remedy this situation, we change the surface that we are working on by blowing up $\P^2$ at the $r$ zeroes $p_1,\ldots,p_r$ to obtain $\pi \colon Z = \P^2(p_1,\ldots,p_r) \to \P^2$.
    
    Hereafter, we follow the notation of \cite[Surfaces de Del Pezzo]{demazure}. The strict transform of $\sV(F)\subset \P^2$ is an irreducible curve on $Z$ defined by a section $s\in H^0(Z,\omega_Z^{\otimes -2})$, where $\omega_Z = \sO_Z(- 3 E_0 + \sum_{i=1}^r E_i)$ is the canonical bundle of $Z$, $E_0$ is the strict transform of a general line and $E_1, \dots, E_r$ are the exceptional divisors over $p_1,\dots, p_r$. Since the points $p_i$ are isolated real zeroes that are nodes of the complex curve, the section $s$ does not vanish on any real point of $Z$. Now we can apply Scheiderer's result \cite[Theorem 4.11]{Scheiderer2011} on $Z$ with, using his notation, $L = M = - \omega_Z$ and $f = g = s$. What we need to check is that the anticanonical divisor on $Z$ is ample, which is implied by the assumption of general position of the zeroes, since $Z$ is a Del Pezzo surface \cite[Chapter 8]{Dolgachev2012}. So, $s^{k} \in H^0(Z,\omega_Z^{\otimes -2k})$ is a sum of squares for all sufficiently large $k$. Since sections in $H^0(Z,\omega_Z^{\otimes -k})$ correspond to curves of degree $3k$ in $\P^2$ that vanish at $p_1,\ldots,p_r$ to order $k$, we can push this representation forward to see that $F^k = \pi_*(s^k)\in H^0(\P^2,\sO_{\P^2}(6k))$ is a sum of squares. Therefore, $F$ is not stubborn. 
\end{ex}

For more special nonnegative sextics $F$, this strategy has two main problems. First of all, the real zeroes do not have to be in general position so that the blow-up is not necessarily a Del Pezzo surface (but rather a weak Del Pezzo surface). Moreover, the zeroes might not be ordinary double points (so the Hessian might have a smaller rank), which then requires an iterated blow-up procedure to get a strictly positive section of some line bundle. These issues can lead to surfaces whose anticanonical bundle is not ample, so that Scheiderer's results do not apply. We will address these problems by a more careful analysis, for which we mainly rely on \cite{demazure}. As an explicit example, in Section 6 of \cite{Stubborn2024} the authors consider a perturbed Motzkin polynomial $M_1(x,y,z)=x^4y^2+x^2y^4+z^6-x^2y^2z^2$ which is not a sum of squares, with two zeroes at $[1:0:0]$ and $[0:1:0]$, each having real delta invariant 3. However, $M_1^3$ is a sum of squares, and thus $M_1$ is not stubborn.

As a first step, we show that the embedded resolution of singularities $\pi\colon X \to \P^2$ of $\sV(F)\subset \P^2$ for a nonnegative sextic $F$ that is not a sum of squares leads to a weak Del Pezzo surface. 
We choose the following basis of the Picard group $\Pic(\P^2(\Phi))$ as in \cite[Surfaces de Del Pezzo II.2]{demazure}, constructed iteratively via $\pi\colon \P^2(\Phi) \to \P^2$ as follows. Let $E_0 = \pi^{-1}(L)$ be the preimage of a general line $L\subset \P^2$ and let $E_1,\ldots,E_r$ be the total transforms of $p_1,\ldots,p_r$. Then the map $\Pic(\P^2(\Phi)) \to \Z^{r+1}$, $\xi\mapsto (\xi.E_0; \xi.E_1,\ldots,\xi.E_r)$ is an isomorphism of lattices. Writing $\omega$ for the canonical class of $\P^2(\Phi)$, we have 
\begin{align}
 & E_0^2 = 1 \ ; \ E_i^2 = -1 \ ; \ E_i.E_j = 0 \text{ for } 0<i\neq j ; \\
 & \omega = - 3 E_0 + \sum_{i=1}^r E_i \mapsto (-3;-1,-1,\ldots,-1) \in \Z^{r+1} \\
 & \omega.E_0 = -3 \ ; \ \omega.E_i = -1 \text{ for } i>0 \ ; \ \omega.\omega = 9 - r.
\end{align}

\begin{lem} \label{lem:ternary_sextics_enriques}
    Let $F\in\R[x,y,z]_6$ be a nonnegative ternary sextic that is not a sum of squares. Let $\pi\colon X \to \P^2$ be an embedded resolution of the real singularities of $\sV(F)\subset \P^2$ obtained as an iterated blow-up $X = \P^2(\Phi)$  for $\Phi = (p_1,\ldots,p_r)$, $\Phi_i = (p_1,\ldots,p_i)$, and $p_{i+1}\in \P^2(\Phi_i)(\R)$.
    Then $\Phi$ is in almost general position in the sense of \cite[Surfaces de Del Pezzo III.2]{demazure}: 
    \begin{enumerate}
        \item for all $i = 1,\ldots,r$, the point $p_i\in \P^2(\Phi_{i-1})$ does not belong to any of the irreducible components of the divisors $E_1,\ldots,E_{i-1}\subset \P^2(\Phi_{i-1})$ which is not equal to some $E_j$.
        \item no line in $\P^2$ passes through $4$ points of $\Phi$;
        \item no irreducible conic in $\P^2$ passes through $7$ points of $\Phi$.
    \end{enumerate}
\end{lem}
To illustrate condition (1), consider the exceptional divisor $E_1\subset \P^2(p_1)$, which is irreducible, and a point $p_2 \in E_1$. We can then blow up $p_2$ and obtain $E_1,E_2\subset \P^2(p_1,p_2)$. The total transform $E_1$ in $\P^2(p_1,p_2)$ has now two irreducible components, the exceptional divisor $E_2$ and the strict transform $\wh{E_1}$ of $E_1$ in $\P^2(p_1)$. In condition (1), we can now choose $p_3$ to be on the irreducible divisor $E_2$, but not on $\wh{E_1}$. 
\begin{proof}
    Suppose that four (distinct) real points of $\Phi$ lie on a line $L$ in $\P^2$. As the sextic $F$ is nonnegative, then $F$ is singular at each of these points so that the intersection multiplicity of $L$ and $\sV(F)$ at each of these points is at least $2$. By B\'ezout's Theorem, this implies that $F$ vanishes identically along the real line $L$. Then \cite[Main Theorem. (A)]{Choi1980RealZO} shows that $F$ is a sum of squares. The same reasoning applies for seven points on an irreducible conic: 
    B\'ezout's theorem again implies that $F$ vanishes identically along a real conic if $F$ has seven zeroes on it. So, (2) and (3) are satisfied.
    
    It remains to check condition (1). Here, we use that any real zero of $F$ is a point of multiplicity $2$ because $F$ is not a sum of squares, see e.g. \cite[Lemma 21]{Stubborn2024}.
    Then the strict transform of $\sV(F)$ also has a zero of multiplicity at most $2$ on any exceptional divisor, see e.g. \cite[Lemma 1.(a)]{Lipman1975}. 
    By induction, this implies condition (1).
\end{proof}

\begin{Def}[{see \cite{Stubborn2024}}]
    \label{def:delta_invariant}
    Recall that the \emph{multiplicity} of a polynomial $f \in \R[x,y]$ at $p\in \mathbb{A}^2$ is the largest integer $m=m_{p}(f)$ such that all the partial derivatives of $f$ of order at most $m-1$ vanish at $p$. 
    Then the \emph{(local) real delta invariant} $\delta_{p}^{\R}(f)$ is defined as
    \begin{align}\label{eq:delta}
        \delta_{p}^{\R}(f) \coloneqq \frac{m_{p}(f)(m_{p}(f)-1)}{2} + \sum_{p'} \delta^{\R}_{p'}(f'),
    \end{align}
    where the sum is taken over all real first-order infinitely near points of $f=0$ at $p$.   
    We define the (local) real delta invariant of a real form $F\in \R[x,y,z]$ at $p\in \P^2(\R)$ by setting $\delta_{p}^{\R}(F):= \delta_{p}^{\R}(f)$, where $f$ is the dehomogenization of $F$ with respect to an affine chart containing $p$. 
    Given the local nature of \eqref{eq:delta}, this definition of $\delta^{\R}_{p}(F)$ does not depend on the choice of an affine chart. Finally, we define the \emph{total real delta invariant} as
    \[
        \delta^{\R}(F) \coloneqq \sum_{p\in \P^2(\R)} \delta_{p}^{\R}(F).
    \]
\end{Def}

The (local) real delta invariant can be seen as an analog of the classical complex one: we refer the reader to \cite{Stubborn2024} for more information. We can now prove our first result on ternary sextics.

\begin{Thm}\label{thm:stubbornSextics}
  Let $F\in\R[x,y,z]$ be a nonnegative ternary sextic that is not a sum of squares. If $\delta^\R(F)\leq 8$, then $F$ is not stubborn. 
\end{Thm}

\begin{proof}
    We adapt the strategy in Example~\ref{ex:genericSextic} to the more degenerate cases relying on \cite[Surfaces de Del Pezzo II-V]{demazure}. Let $\pi \colon Z = \P^2(\Phi) \to \P^2$ be the embedded resolution of the real singularities of $\sV(F)\subset \P^2$. 
    Since $F$ is nonnegative and not a sum of squares, it follows that $Z$ is a weak Del Pezzo surface, see Lemma~\ref{lem:ternary_sextics_enriques} and \cite[Cor.~8.1.24]{Dolgachev2012}, and that the strict transform of $\sV(F)$ in $Z$ has no real points. Since $F$ is a section of $\omega_{\P^2}^{\otimes -2}$ and $Z$ is a resolution of the real singularities of $\sV(F)$, it gives a section $f$ of the line bundle $\omega_Z^{\otimes  -2}$ corresponding to the strict transform of $\sV(F)$ in $Z$. 
    
    Let $\phi \colon Z \to \overline{Z}$ be the map to the anticanonical model of $Z$ (see \cite[Surfaces de Del Pezzo V]{demazure}). The map $\phi$ induces canonical isomorphisms $\phi_*(\omega_Z^{\otimes n}) = \omega_{\overline{Z}}^{\otimes n}$ and $\phi^*(\omega_{\overline{Z}}^{\otimes n}) = \omega_Z^{\otimes n}$ where the invertible sheaf $\omega_{\overline{Z}}$ is defined as the push forward $\phi_*(\omega_Z)$ of the canonical bundle of $Z$, see \cite[Surfaces de Del Pezzo V, Th\'eor\`eme 2(d)]{demazure}. The section $f$ of $\omega_Z^{\otimes -2}$ then gives a section $\phi_*(f)$ of $\omega_{\overline{Z}}^{\otimes -2}$.
    
    We now show that this section has no real zero on $\overline{Z}$ using \cite[Surfaces de Del Pezzo V, Th\'eor\`eme 2]{demazure}. This theorem shows that $\phi_*(f)$ has no zeroes outside the singularities of $\overline{Z}$ since $\phi$ is an isomorphism between $U\subset Z$ and $\overline{Z}_{reg}$, by parts (a) and (c) of the above theorem (where $U$ is the complement of the fundamental cycles, see \cite[p.~54]{demazure}). Now, the vanishing of the section $\phi_*(f)$ at a singularity of $\overline{Z}$ means that $f$ would have to vanish on the fundamental cycle being mapped to the singularity. This, however, cannot happen due to \cite[Surfaces de Del Pezzo IV, Corollaire 1, page 53]{demazure}: 
    if $f$ had a zero on a fundamental cycle $\Gamma$, then it would vanish entirely along $\Gamma$. However, $f$ defines the strict transform of $\sV(F)$, which is irreducible and cannot contain an exceptional divisor as a component. 
    
    From the previous paragraph we deduce that $\phi_*(f)$ is a positive section of $\omega_{\overline{Z}}^{\otimes -2}$, which is ample. We can then use Scheiderer's result \cite[Theorem 4.1]{Scheiderer2011} (with $X = Z$, $L=M = \omega_{\overline{Z}}$, $r=0$, and $f = g = \phi_*(f)$), which shows that there is an odd $k\in \N$ such that $\phi_*(f)^k\in H^0\left( \overline{Z},\omega_{\overline{Z}}^{\otimes -2k}\right)$ is a sum of squares of sections of $\omega_{\overline{Z}}^{\otimes -k}$. By the isomorphism $\phi^*$, we can pull this representation back to show that $f^k$ is a sum of squares. Finally, this representation of $f^k\in H^0(Z,\omega_Z^{\otimes -2k})$ as a sum of squares can be pushed forward through the resolution of singularities to give a sum of squares representation of $F^k = \pi_*(f^k)$. 
\end{proof}

The following example illustrates the role of fundamental cycles for the anticanonical model as in the previous proof.
\begin{ex}
    Consider a ternary sextic that near $(0:0:1)$ looks like $x^2 + y^4 = 0$. To resolve the singularity of the corresponding curve, we need to do two consecutive blowups. Blowing up $(0:0:1)\in \P^2$ in coordinates $a_1$ and $b_1 = y$ with $x = a_1 b_1$ transforms the local equation $x^2 + y^4 = 0$ to $(a_1 b_1)^2 + b_1^4 = 0$, which factors as $b_1^2 ( a_1^2 + b_1^2)$. The exceptional divisor on this blow-up $X_1$ is defined in those coordinates by $b_1 = 0$ and the form $(a_1^2 + b_1^2)$ has a double root on this exceptional divisor $E_1$. So we blow up the surface $X_1$ at the point $(0,0)$ in the given coordinates again. Denote the new surface by $X_2$. In local coordinates $a_2$ and $b_2 = b_1$ with $a_1 = a_2b_2$ the equation becomes $0 = a_1^2 + b_1^2 = (a_2b_2)^2 + b_2^2 = b_2^2 ( 1 + a_2^2)$. The new exceptional divisor is defined by $b_2 = 0$ and the form $1 + a_2^2$ has no real zero (two complex conjugate zeroes). We have resolved the singularity. 

    Now the final surface $X_2$ is a Del Pezzo surface whose Picard group $\Pic(X_2)$ is $\Z^3$. Demazure chooses a basis with intersection form $\mathrm{diag}(1,-1,-1)$ in \cite[Surfaces de Del Pezzo II]{demazure}. The \emph{fundamental cycles} are the effective divisors $D$ on $X_2$ with the property $D^2 = -2$ and $D.\omega_{X_2} = 0$. There are only two solutions to these two equations, namely $(0,1,-1)$ and $(0,-1,1)$. Only one of the two is effective, namely $(0,1,-1)$ which is the strict transform $\wh{E_1} = E_2 - E_1$ of the first exceptional divisor $E_1$ under the second blow-up. 

    These fundamental cycles play a special role for the anticanonical model of weak Del Pezzo surfaces \cite[Surfaces de Del Pezzo V.2]{demazure}. The strict transform of our sextic does not intersect this divisor $\wh{E_1}$ anymore. It did intersect $E_1$ in a point, namely in a simple node and the second blow-up removed this intersection point. Therefore, the push forward of the divisor defined by our ternary sextic to the anticanonical model of the weak Del Pezzo surface does not contain the singularity corresponding to the contraction of the fundamental cycle.
\end{ex}

\subsection{Sextics with many zeroes are stubborn}

We now present a proof that a nonnegative ternary sextic that is not a sum of squares and has sufficiently many zeroes must be stubborn. Thus we complete our classification of stubborn and non-stubborn ternary sextics via their number of real zeroes.
\begin{Thm}
    \label{thm:delta_geq_9}
    Let $F\in\R[x,y,z]$ be a nonnegative ternary sextic that is not a sum of squares. If $\delta^\R(F)\geq 9$, then $F$ is stubborn.
\end{Thm}

\begin{proof}
    In this proof, it is more convenient to use the language of divisors instead of line bundles.
    Let $\pi\colon Z = \P^2(p_1,\ldots,p_r) \to \P^2$ be an embedded resolution of the real singularities of $\sV(F)\subset \P^2$. Since $\delta^\R(F)\geq 9$, we have $r\geq 9$. 
    Indeed, since $F$ is a nonnegative ternary sextic that is not a sum of squares, we have $m_{p}(F)=2$ for any real zero $p\in \P^2(\R)$ of $F$. 
    If $\delta_{p}^{\R}(F)>1$, then by \eqref{eq:delta} the strict transform of $\mathcal{V}(F)$ intersects the exceptional fiber over $p$ in one point $p'$ where it has multiplicity $2$. In other words, we have $\delta^{\R}_{p}(f)=1+\delta^{\R}_{p'}(f')$ and $m_{p'}(f')=2$, cf. the proof of \cite[Lemma 26]{Stubborn2024}. Proceeding inductively, we see that $p\in \P^2$ with $\delta_{p}^{\R}(F)=k$ contributes with $k$ many points to $Z = \P^2(p_1,\ldots,p_r)$.
    Passing now to the global quantity $\delta^{\R}(F)$, we see that there must be at least $9$ points in $Z$.

    We first show $h^0(Z, 3E_0 -\sum_{i=1}^9 E_i) \geq 1$. This follows from the Riemann-Roch formula \cite[Theorem V.1.6]{zbMATH03572315}, here for $r=9$: \[ h^0(Z,D) = \frac12 D.(D-K) + 1 + p_a(Z) + h^1(Z, D) - h^0(Z,K-D)\]
    with $D = 3E_0 - \sum_{i=1}^9 E_i$. We have that $h^0(Z,D) \geq 1$, since $p_a(Z) = 0$, $h^1(Z, D) \geq 0$, $h^0(Z,K-D) = 0$ (since it intersects a general line in a negative number of points), and $D.(D-K) = 0$. To compute this last intersection product, we first expand
     \[ D.(D-K) = \left(3E_0 - \sum_{i=1}^9 E_i\right).\left(6E_0 - 2 \sum_{i=1}^9 E_i\right) = 18 + 2 \cdot \sum_{i=1}^9 E_i^2  = 0\]
    using $E_0.E_i = 0$, $E_i^2 = -1$ and $E_i.E_j = 0$ for $i,j\geq 1$, $i\neq j$. So we obtain $h^0(Z,D) \geq 1$ as desired.

    If $r=10$, the only change above is that $D-K$ becomes $6E_0 - 2\sum_{i=1}^9 E_i - 2 E_{10}$, which intersects $D = 3E_0 - \sum_{i=1}^9 E_i$ the same way because $E_{10}.E_i = 0$ for $1\le i \le 9$ and $E_0.E_{10} = 0$. Hence $h^0(Z,D) \geq 1$ holds true also in this case.
    
  So now pick a nonzero section $c\in H^0(Z,3E_0 - \sum_{i=1}^9 E_i)$. Its pushforward $\pi_*(c)$ corresponds to a plane cubic $C$.
  Since $F$ is irreducible, the curves $\sV(C)$ and $\sV(F)$ intersect in $18$ points in $\P^2$, counted with multiplicity. Since the strict transforms of these curves do not intersect on $Z$ (by the above calculation $D.(D-K) = 0$), all these intersection points must be real. In fact, all of them are non-singular points on the cubic: if the cubic was singular at a point $p_j$, then its strict transform would have class $3E_0 - \sum_{i\neq j} E_i - 2 E_j$, which, by a similar computation as above, gives a negative intersection $-2$ with the class $6E_0-2\sum_{i=1}^9 E_i$ of the strict transform of $\sV(F)$. This is a contradiction to B\'ezout's Theorem on $Z$. 

    We finish now by showing that the polynomial $F$ is stubborn in the homogeneous coordinate ring of the cubic curve. First we show that it is not a sum of squares. Suppose it is; since the polynomial $F$ has the maximal number of real roots on the cubic $\sV(C)$, it is an extreme ray of the cone of sums of squares, hence a square (cf. Section \ref{sec:cubics}). 

    Write $F = A^2 + CB$ for cubics $A,B\in \R[x,y,z]$. Since $F(p_1) = 0$ and $C(p_1) = 0$, it follows that $A(p_1) = 0$ as well, so that the strict transform of $\sV(A)$ in $\P(p_1)$ has class $3L-E_1$. Applying the same reasoning to $p_2$ and continuing inductively, the strict transform of $\sV(A)$ in $Z = \P(p_1,\ldots,p_9)$ has class $3E_0-\sum_{i=1}^9 E_i$. B\'ezout's Theorem on $Z$ implies that the strict transforms of $\sV(C)$ and $\sV(A)$ have no intersection points, which means that $\sV(C)$ and $\sV(A)$ are a complete intersection in $9$ points, counted with multiplicity, in $\P^2$. Therefore, the two cubics $A$ and $C$ generate the ideal of this scheme of length $9$. Concretely, this implies that $B = \alpha A + \gamma C$ is a linear combination of $A$ and $C$. 

    Now consider the rational map $\phi\colon \P^2 \dashrightarrow \P^1$, $(x:y:z)\mapsto (A(x,y,z):C(z,y,z))$. Since the fiber over a general point is a cubic curve, the fiber over a generic real point contains a real point. This shows that $F = A^2 + CB$ corresponds to a quadratic form on $\P^1$ that is globally nonnegative and hence a sum of squares in $A$ and $C$, that is, in $\R[x,y,z]$, which contradicts the assumption that $F$ is not a sum of squares. 

    Hence $F$ is real-rooted, nonnegative, and not a square on $C$. Then \Cref{cor:plane_cubics} shows that $F$ is stubborn on $C$, and we can use \Cref{prop:inherit_stubbornness} to conclude the proof.
\end{proof}

\section{Questions and Conjectures}
We finish by stating some open questions and directions of research. Motivated by \cite{Stubborn2024} which shows that the set of non-stubborn polynomials forms a convex cone containing the interior of $P_X$, we make the following conjecture. 
\begin{Conj}
Let $X\subset \P^n$ be a totally real variety and $\mathcal{F}$ a face of $P_X$. Then either all points in the relative interior of $\mathcal{F}$ are stubborn or none of them are. 
\end{Conj}
For smooth curves, this follows from \Cref{thm:main-curves} (only extreme rays with a maximal number of real zeroes are stubborn). Our next conjecture deals with (nonexistence) of gaps in stubborn degrees.

\begin{Conj}
Let $X\subset \P^n$ be a totally real variety. There exists a positive integer $d$ such that for all $1\leq k<d$, there are no stubborn forms of degree $2k$ on $X$ and for all $k\geq d$ there exist stubborn forms of degree $2k$ on $X$.
\end{Conj}

Motivated by the discussion on lifting nonnegativity in \Cref{sec:lifting} we ask a general question, which we think opens an intriguing research direction:
\begin{question}
Classify all pairs of totally real varieties $X\subset Y$ such that any polynomial $f\in P_X$ can be lifted to a nonnegative polynomial $Y$.
\end{question}

Varieties for which all nonnegative quadratics are sums of squares have been classified in \cite{Blekherman2013SumsOS}. We ask for a classification of varieties on which there are no stubborn quadrics.

\begin{question}
Classify all totally real varieties $X$ such that there are no stubborn quadrics on $X$.
\end{question}

We end with a more concrete question. In \cite[Conj. 6]{Stubborn2024} it is conjectured that a globally nonnegative form (that is not a square) spanning an extreme ray of $P_{\,\P^n,2d}$  is stubborn.
This holds for ternary sextics $P_{\,\P^2, 6}$ by \cite[Thm. 1]{Stubborn2024}, in which case extreme rays (that are not squares) are spanned
by forms $F\in P_{\,\P^2,6}$ satisfying $\delta^{\R}(F)=10$, that is, $F$ is a nonnegative degree $6$ form with $10$ real zeros counted with their multiplicities \eqref{eq:delta}, see \cite[Thm. 2]{blekherman_hauenstein_ottem_ranestad_sturmfels_2012} and \cite[Lemma 23]{Stubborn2024}.
The next natural case is the cone of quaternary quartics $P_{\,\P^3, 4}$. Exposed extreme rays (that are not squares) in this case are spanned by $F\in P_{\,\P^3,4}$ with exactly $10$ real zeros.
Furthermore, it follows from \cite[Thm. 3]{blekherman_hauenstein_ottem_ranestad_sturmfels_2012} that $F=\det(xA+yB+zC+wD)$ is the determinant of a symmetric $4\times 4$ matrix of linear forms.
Thus, it is natural to ask whether the proof idea of \cite[Thm. 1]{Stubborn2024} can be adapted to this case.
\begin{question}
Is every non-square nonnegative form that spans an extreme ray of $P_{\,\P^3,4}$ stubborn?
\end{question}
A more general conjecture is presented in \cite[Conj. 6]{Stubborn2024}, but this is the first instance in which to start, and which is better understood. 

\bibliography{biblio_arXiv}

\end{document}